\documentclass[11 pt]{article}

\title{Integrality in the Steinberg module and\\ the top-dimensional cohomology of $\SL_n \O_K$}
\author{Thomas Church,\ Benson Farb,\ and Andrew Putman\thanks{TC was supported by NSF grants DMS-1103807 and DMS-1350138 and the Alfred P.\ Sloan Foundation; BF was supported by NSF grant DMS-1105643; AP was supported by NSF grant DMS-1255350 and the Alfred P.\ Sloan Foundation.}\vspace{-6pt}}

\usepackage{amsmath,amssymb,amsthm,amscd,amsfonts}
\usepackage{epsfig,pinlabel}
\usepackage[hmargin=1.2in, vmargin=0.90in]{geometry}
\usepackage{type1cm}
\usepackage[strings]{underscore} 
\usepackage{calc}
\usepackage[shortlabels]{enumitem}
\usepackage{bm}
\usepackage{microtype}
\usepackage{hyperref}

\newtheorem{innerthm}{Theorem}
\newenvironment{theorem-prime}[1]{\renewcommand\theinnerthm{\ref*{#1}$'$}\innerthm}{\endinnerthm}

\theoremstyle{plain}
\newtheorem{theorem}{Theorem}[section]
\newtheorem{maintheorem}{Theorem}
\newtheorem{proposition}[theorem]{Proposition}
\newtheorem*{mainthm-notconn}{Theorem~\ref{theorem:B2notconnected}}
\newtheorem{lemma}[theorem]{Lemma}

\newtheorem{claims}{Claim}

\newcommand\BeginClaims{\setcounter{claims}{0}}

\theoremstyle{definition}
\newtheorem{remark}[theorem]{Remark}

\newtheorem{definition}[theorem]{Definition}

\newtheorem*{notation}{Notation}

\newtheorem{example}[theorem]{Example}

\numberwithin{equation}{section}

\DeclareMathOperator{\Sp}{Sp}
\DeclareMathOperator{\GL}{GL}
\DeclareMathOperator{\SL}{SL}

\newcommand\HBolic{\ensuremath{\mathbb{H}}}

\newcommand\R{\ensuremath{\mathbb{R}}}
\newcommand\C{\ensuremath{\mathbb{C}}}
\newcommand\Z{\ensuremath{\mathbb{Z}}}
\newcommand\Q{\ensuremath{\mathbb{Q}}}
\newcommand\N{\ensuremath{\mathbb{N}}}

\newcommand\F{\ensuremath{\mathbb{F}}}

\DeclareMathOperator{\HH}{H}

\DeclareMathOperator{\Aut}{Aut}

\newcommand\Span[1]{\ensuremath{\langle #1 \rangle}}

\newcommand\Set[2]{\ensuremath{\{\text{#1 $|$ #2}\}}}

\newcommand{\para}[1]{\bigskip\noindent\textbf{#1.}}
\newcommand{\parasmallskip}[1]{\smallskip\noindent\textbf{#1.}}
\newcommand{\paranoskip}[1]{\noindent\textbf{#1.}}
\renewcommand{\epsilon}{\varepsilon}

\renewcommand{\O}{\mathcal{O}}
\newcommand\PartialBases{\ensuremath{\mathcal{B}}}
\newcommand{\PB}{\PartialBases}
\newcommand{\PartialBasesGood}{\PartialBases^{\text{gd}}}
\newcommand{\PBg}{\PartialBasesGood}
\DeclareMathOperator{\vcd}{vcd}
\DeclareMathOperator{\class}{cl}
\newcommand{\cl}{\class}

\newcommand\Tits{\ensuremath{\mathcal{T}}}
\newcommand\T{\Tits}
\DeclareMathOperator{\St}{St}

\newcommand\LL{\mathbf{L}}
\newcommand\II{\mathbf{I}}
\newcommand\vv{\mathbf{v}}
\newcommand\Ttop{\widetilde{T}}
\DeclareMathOperator{\Link}{Link}
\newcommand\Poset{\ensuremath{\mathcal{P}}}
\newcommand\PP{\Poset^{(n-2)}}
\DeclareMathOperator{\height}{ht}
\newcommand\VV{\mathbf{V}}
\newcommand\bS{\mathbf{S}}
\newcommand\into{\hookrightarrow}
\newcommand\onto{\twoheadrightarrow}
\newcommand\coloneq{\mathrel{\mathop:}\mkern-1.2mu=}
\newcommand\e{\mathbf{e}}
\newcommand\CC{\mathbf{C}}
\newcommand\abs[1]{\left\lvert#1\right\rvert}
\newcommand\tensor{\otimes}
\DeclareMathOperator{\Int}{Int}
\DeclareMathOperator{\Rank}{rk}

\DeclareMathOperator\spn{span}
\DeclareMathOperator\ord{ord}

\newcommand\iso{\cong}
\newcommand\GLvcd{\nu_n'}
\newcommand\SLvcd{\nu_n}

\newcommand{\arXiv}[1]{\href{http://arxiv.org/abs/#1}{arXiv:#1}}
\newcommand{\doi}[1]{\href{http://dx.doi.org/#1}{doi:#1}}
\newcommand{\myemail}[1]{\href{mailto:#1}{\nolinkurl{#1}}}

\newcommand{\change}{}

\date{February 2, 2019}

\begin{document}

\maketitle

\begin{abstract}
We prove a new structural result for the spherical Tits building attached to $\SL_n K$ for many number fields $K$, and more generally for the fraction fields of many Dedekind domains $\O$:  the Steinberg module $\St_n(K)$ is generated by integral apartments if and only if the ideal class group $\cl(\O)$ is trivial.    We deduce this integrality by proving that the complex of partial bases of $\O^n$ is Cohen--Macaulay.  
We apply this to prove new vanishing and nonvanishing results for $\HH^{\SLvcd}(\SL_n\O_K;\Q)$, where $\O_K$ is the ring of integers in a number field and $\SLvcd$ is the virtual cohomological dimension of  $\SL_n\O_K$.   The (non)vanishing depends on the (non)triviality of the class group of $\O_K$.    We also obtain a vanishing theorem 
for the cohomology $\HH^{\SLvcd}(\SL_n\O_K;V)$ with twisted coefficients $V$. 
\end{abstract}

\section{Introduction}
\label{section:introduction}

\subsection{The Tits building and the Steinberg module}

One of the most fundamental geometric objects attached to the general linear group $\GL_n K$ over a field $K$ is its associated \emph{Tits building}, denoted $\Tits_n(K)$.  The space $\Tits_n(K)$ is the $(n-2)$-dimensional
simplicial complex whose $p$-simplices are flags of subspaces
\[0 \subsetneq V_0 \subsetneq \cdots \subsetneq V_p \subsetneq K^n.\]
The group $\GL_n K$ acts on $\Tits_n(K)$ by simplicial automorphisms.  
By the Solomon--Tits theorem \cite{SolomonChar}, the space $\Tits_n(K)$ is homotopy 
equivalent to a wedge of $(n-2)$-dimensional spheres, so it has
only one interesting homology group.

\begin{definition}[\textbf{The Steinberg module}]
\label{def:Steinberg}
The \emph{Steinberg module} $\St_n(K)$ 
is the $\GL_n K$-module given by the top homology group $\widetilde{\HH}_{n-2}(\Tits_n(K);\Z)$.
\end{definition}

The Solomon--Tits theorem also gives a generating set for $\St_n(K)$.
A \emph{frame} for $K^n$ is a  set $\LL = \{L_1,\ldots,L_n\}$ of lines in $K^n$ such that
$K^n = L_1 \oplus \cdots \oplus L_n$.  The \emph{apartment} $A_{\LL}$ corresponding to $\LL$ is the full subcomplex of $\Tits_n(K)$ on the $2^n-2$ subspaces $\spn(L_i \,|\, i \in I)$ for $\emptyset\subsetneq I\subsetneq \{1,\ldots,n\}$.
The apartment $A_{\LL}$ is 
homeomorphic to an $(n-2)$-sphere,  
so its fundamental class determines an {\em apartment class}
\[[A_{\LL}] \in \widetilde{\HH}_{n-2}(\Tits_n(K);\Z) = \St_n(K)\]
which is defined up to $\pm 1$.  The Solomon--Tits theorem 
states that $\St_n(K)$ is generated by the set
of apartment classes. 

\subsection{Integrality and non-integrality}  
Consider a Dedekind domain $\O$ with field of fractions $K$. While the action of $\GL_n K$ on $\Tits_n(K)$ is transitive, the action of $\GL_n \O$ is usually not, and indeed this action encodes arithmetic information about $\O$.  For example, it is well-known that the number of orbits of the $\GL_2 \O$-action on $\T_2(K)$ is the class number $\abs{\class(\O)}$ of $\O$ (in fact, the orbits are naturally in bijection with the ideal class group $\class(\O)$; see Proposition~\ref{proposition:quotientdesc} below for a generalization).  
In this context, we have the following natural notion.

\begin{definition}[\textbf{Integral apartment}]
\label{def:integral}
Let $\O$ be a Dedekind domain with field of fractions $K$. 
A frame $\LL=\{L_1,\ldots,L_n\}$ of $K^n$ is \emph{integral}, and $A_{\LL}$ is an \emph{integral apartment}, if 
\[\O^n=(L_1\cap \O^n)\oplus \cdots\oplus (L_n\cap \O^n).\]
\end{definition}
In the case $\O=\Z$, and more generally when $\O$ is Euclidean, Ash--Rudolph~\cite[Theorem~4.1]{AshRudolph} proved that $\St_n(K)$ is generated by the fundamental classes of integral apartments.
Our first theorem extends this 
result of Ash--Rudolph to a wide class of Dedekind domains of arithmetic type, whose definition we recall.

\begin{definition}[\textbf{Dedekind domain of arithmetic type}]
 Let $K$ be a global field, i.e.\ either a number field ($[K:\Q]<\infty$) or a function field in one variable over a finite field ($[K:\F_q(T)]<\infty$). Let $S$ be a finite nonempty set of places of $K$; if $K$ is a number field, assume that $S$ contains all infinite 
places. The ring of $S$-integers \[\O_S\coloneq\{x\in K \,|\, \ord_{\mathfrak{p}}(x)\geq 0\text{ for all }\mathfrak{p}\not\in S\}\] is a Dedekind domain, and we say that $\O_S$ is a Dedekind domain of \emph{arithmetic type}. 
\end{definition}


\begin{maintheorem}[\textbf{Integrality Theorem}]
\label{thm:ashrudolph}
\label{maintheorem:integralapartments}
Let $\O$ be a Dedekind domain with $\abs{\class(\O)}=1$ and field of fractions $K$. Assume either that $\O$ is Euclidean, or that $\O=\O_S$ is a Dedekind domain of arithmetic type such that $\abs{S}>1$ and $S$ contains a non-complex place.  Then $\St_n(K)$ is spanned by integral apartment classes.
\end{maintheorem}

Under the conditions of Theorem~\ref{thm:ashrudolph}, the 
set of 
integral apartments in $\Tits_n(K)$ is precisely the $\GL_n\O$-orbit of a single standard apartment.  Thus Theorem~\ref{thm:ashrudolph} implies that $\St_n(K)$, which is obviously cyclic as a $\GL_n K$-module, is in fact cyclic as a $\GL_n \O$-module.

\begin{remark}
Ash--Rudolph's proof \cite{AshRudolph} of Theorem~\ref{maintheorem:integralapartments}
when $\O$ is Euclidean is based on a beautiful generalization of the method of continued fractions to higher dimensions.  Using the Euclidean function on
$\O$ as a measure of ``complexity'',  they give an algorithm to write a non-integral apartment class as a sum of integral apartment classes.  Our proof is quite different, even in the special case that $\O$ is Euclidean: non-integral apartments never actually
show up in our proof, and we do not even make use of the fact that
$\St_n(K)$ is generated by apartment classes.  
\end{remark}

By Hasse--Chevalley's generalization of Dirichlet's theorem, the group of units $\O_S^\times$ has rank $\abs{S}-1$, so the assumption $\abs{S}>1$ is equivalent to $\abs{\O_S^\times}=\infty$. The assumptions on $S$ rule out 
only two families of Dedekind domains of arithmetic type, for which Theorem~\ref{thm:ashrudolph} need not hold (see Remark~\ref{rem:thmCfalseimaginary} below):
\begin{itemize}[nosep]
\item the ring of integers $\O_K$ in a totally imaginary number field $K$;
\item those $\O_S$ in positive characteristic with finitely many units, such as $\F_q[T]$. 
\end{itemize}
The assumption in Theorem~\ref{thm:ashrudolph} that $\abs{\class(\O_S)}=1$ obviously excludes many more examples.  However, this assumption is necessary in a strong sense, as the following theorem shows.

\begin{maintheorem}[\textbf{Non-Integrality Theorem}]
\label{maintheorem:nonintegrality}
\label{thm:Stnotgen}
Let $\O$ be a Dedekind domain with field of fractions $K$.  If  $1<\abs{\class(\O)}<\infty$ and $n\geq 2$ 
then $\St_n(K)$ is not generated by integral apartment classes.
\end{maintheorem}
Theorem~\ref{maintheorem:nonintegrality}
 applies in particular to all number rings with nontrivial class group.

\subsection{The top-degree cohomology of \texorpdfstring{$\SL_n \O_K$}{SLn(O)} and \texorpdfstring{$\GL_n \O_K$}{GLn(O)}}
\label{section:topcohomologyintro}

When $\O_K$ is the ring of integers in a number field $K$, the Steinberg module $\St_n(K)$ is directly connected to the cohomology of  $\SL_n \O_K$ and $\GL_n \O_K$, as we now explain.   

Let $r_1$ (resp.\ $2r_2$) be the number of real (resp.\ complex) embeddings of $\O_K$.  Then  
$\O_K\tensor \R\iso \R^{r_1}\oplus \C^{r_2}$ and $\SL_n (\O_K\tensor \R) \iso (\SL_n \R)^{r_1} \times (\SL_n \C)^{r_2}$.  Recall that the symmetric space associated to such a Lie
group is the quotient of the group by a maximal compact subgroup.  The quotient $M_K$ of the symmetric space associated to $\SL_n (\O_K\tensor \R)$ by the discrete group $\SL_n\O_K$ is a  Riemannian orbifold with $\HH^*(M_K;\Q)\iso H^*(\SL_n\O_K;\Q)$.

The computation of these cohomology groups is a fundamental problem in group theory, topology and number theory. The existence of $M_K$ implies that $\HH^i(\SL_n\O_K;\Q)=0$ for $i>\dim M_K$. But since $M_K$ is not compact,  it is not even clear what the largest $i$ for which $\HH^i(\SL_n\O_K;\Q)\neq 0$ is.   The first progress on this basic problem came in 1973 by Borel and Serre \cite{BorelSerreCorners}. Recall that the {\em virtual cohomological dimension} of a virtually torsion-free group $\Gamma$ is
\[\vcd(\Gamma)\coloneq  \max \Set{$k$}{$\HH^k(\Gamma;V\tensor \Q) \neq 0$ for some $\Gamma$-module $V$}.\]
Set $\SLvcd = \vcd(\SL_n \O_K)$ and $\GLvcd = \vcd(\GL_n \O_K)$.
By constructing and analyzing a compactification of the orbifold $M_K$, Borel--Serre proved that 
\begin{equation}
\label{eq:GLSLvcd}
\GLvcd = r_1\binom{n+1}{2} + r_2\cdot n^2 - n \quad \text{and} \quad \SLvcd = \GLvcd-(r_1+r_2-1).
\end{equation}
In particular $\HH^k(\SL_n \O_K;\Q)=0$ for $k>\SLvcd$. 
The coefficient module $V$ used by Borel--Serre to certify that $\HH^{\SLvcd}(\SL_n \O_K;V)\neq 0$ is infinite-dimensional.  This left open the question as to whether or not $\SL_n \O_K$ has untwisted rational cohomology in the top dimension $\SLvcd$.

In 1976 Lee--Szczarba~\cite[Theorem~1.3]{LeeSzczarbaCongruence} answered this question in the special case that $\O_K$ is Euclidean, proving that $\HH^{\SLvcd}(\SL_n \O_K;\Q)=0$ for these $\O_K$.      Except for some explicit computations in low-dimensional cases (sometimes with the help of computers), there seems to have been no progress on this question since Lee--Szczarba.


Our next theorem extends the theorem of Lee--Szczarba from Euclidean number rings to many number rings with class number $1$ (we remark that even in the Euclidean case our proof is different from
that of Lee--Szczarba).  In addition to untwisted coefficients $\Q$, it also gives a result for twisted coefficient systems arising from rational representations
of the algebraic group $\GL_n$.  
For $\lambda = (\lambda_1,\ldots,\lambda_n) \in \Z^n$ with
$\lambda_1 \geq \cdots \geq \lambda_n$, 
let $V_{\lambda}$ be the rational $\GL_n K$-representation
with highest weight $\lambda$, and define $\|\lambda\| = \sum_{i=1}^{n} (\lambda_i - \lambda_n)$.

\begin{maintheorem}[\textbf{Vanishing Theorem}]
\label{maintheorem:vanishing}
Let $\O_K$ be the ring of integers in an algebraic number field $K$ with $\abs{\class(\O_K)} = 1$.  Suppose that 
$K$ has a real embedding or that $\O_K$ is Euclidean. Then 
\[\HH^{\SLvcd}(\SL_n \O_K;V_{\lambda}) = \HH^{\GLvcd}(\GL_n \O_K;V_{\lambda}) = 0\]
for $n \geq 2 + \|\lambda\|$.  In particular, $\HH^{\SLvcd}(\SL_n \O_K;\Q) = \HH^{\GLvcd}(\GL_n \O_K;\Q) = 0$ for $n\geq 2$.
\end{maintheorem}
For twisted coefficients, Theorem~\ref{maintheorem:vanishing} seems to be new even in the case when $\O_K$ is Euclidean.
As we discuss in Remark~\ref{rem:thmCfalseimaginary} below, the assumption that $K$ has a real embedding is necessary if $\O_K$ is not Euclidean.

In contrast to Theorem~\ref{maintheorem:vanishing}, a different phenomenon occurs when $\abs{\class(\O_K)} \neq 1$.

\begin{maintheorem}[\textbf{Non-Vanishing Theorem}]
\label{maintheorem:nonvanishing} 
Let $\O_K$ be the ring of integers in an algebraic number field $K$.  Then for $n \geq 2$,
\[\dim \HH^{\SLvcd}(\SL_n \O_K;\Q) \geq (\abs{\class(\O_K)} - 1)^{n-1}.\]
\end{maintheorem}

Theorem~\ref{maintheorem:nonvanishing} is classical for $n=2$; see
Example~\ref{ex:hyperbolic3manifolds}. For $n\geq 3$ the result and method of proof are new.  The
analogue of Theorem \ref{maintheorem:nonvanishing} for $\GL_n \O_K$ holds in some cases but not in others; see the forthcoming paper \cite{PutmanStudenmund} for more details.

\begin{remark}
The cohomology of arithmetic groups such as $\SL_n\O_K$ has been fruitfully studied via the theory of automorphic forms,
see e.g.\ \cite{FrankeMain,FrankeSchwermer}. However, it can be difficult to
extract concrete statements from them. For example, it is not clear how to see from \cite{FrankeMain} even Borel--Serre's
result that the rational cohomology of $\SL_n\O_K$ vanishes above its vcd, let alone the sharper vanishing of
Theorem~\ref{maintheorem:vanishing}.

Franke~\cite{FrankeTopological} has also given a more practicable description of the Hecke-trivial subspace of $H^*(\SL_n\O_K;\Q)$. 
For example, for a quadratic imaginary field $K=\Q(\sqrt{-d})$, it can be deduced from \cite[Theorem~9]{FrankeTopological} that $H^i(\SL_n\O_k;\Q)$ contains no Hecke-trivial cohomology for $i>\SLvcd-n+1$. 
But as Theorem~\ref{maintheorem:nonvanishing} shows, this cannot capture all the cohomology; indeed, the Hecke action on the classes we construct in Theorem~\ref{maintheorem:nonvanishing} will be twisted by nontrivial ideal class characters, and thus is not captured by this result of Franke.
\end{remark}
%
%

\subsection{Relationship between Integrality and Vanishing/Non-Vanishing}

Theorems~\ref{maintheorem:vanishing} (Vanishing) and~\ref{maintheorem:nonvanishing} (Non-Vanishing) relate to the Integrality Theorem (Theorem~\ref{thm:ashrudolph}) as follows.  While $\SL_n \O_K$ does not satisfy Poincar\'{e} duality, Borel--Serre~\cite{BorelSerreCorners} proved that it does satisfy \emph{Bieri--Eckmann duality} with \emph{rational dualizing module} $\St_n(K)$.  This means that for any 
$\Q\SL_n \O_K$-module $V$ and any $i\geq 0$ we have
 \[\HH^{\SLvcd-i}\big(\SL_n \O_K;\,V\big) \iso \HH_i\big(\SL_n \O_K;\, \St_n(K)\otimes V\big).\]
This implies in particular that 
\begin{equation}
\label{eq:dualizing}
\HH^{\SLvcd}(\SL_n \O_K;\Q) \iso \HH_0(\SL_n \O_K;\St_n(K)\otimes \Q) \iso (\St_n(K)\otimes\Q)_{\SL_n \O_K}.
\end{equation}
Equation~\eqref{eq:dualizing} converts the problem of computing $\HH^{\SLvcd}(\SL_n \O_K;\Q)$ to the problem of understanding the $\SL_n \O_K$-action on $\St_n(K)$. This understanding is precisely what is given by the Integrality Theorem, which
is applied directly to prove Theorem~\ref{maintheorem:vanishing} for $\SL_n \O_K$. The work of Borel--Serre also implies
that $\GL_n \O_K$ satisfies rational Bieri--Eckmann duality, with a dualizing module that is closely related
to $\St_n(K)$.  Using this, we deduce the $\GL_n \O_K$ case of Theorem~\ref{maintheorem:vanishing} from the
case of $\SL_n \O_K$.


As for Theorem~\ref{maintheorem:nonvanishing}, the projection $\Tits_n(K) \rightarrow \Tits_n(K)/\GL_n \O_K$
induces a map 
\begin{equation}
\label{eqn:provesurj}
\St_n(K) \rightarrow \widetilde{\HH}_{n-2}(\Tits_n(K)/\GL_n\O_K;\Z) \iso \Z^N\qquad\text{where } N=(\abs{\class(\O_K)} - 1)^{n-1};
\end{equation}
see Proposition~\ref{proposition:quotientdesc}.  
By \eqref{eq:dualizing}, this map rationally factors through $\HH^{\SLvcd}(\SL_n\O_K;\Q)$, so to prove Theorem~\ref{maintheorem:nonvanishing} it is enough to prove that the map \eqref{eqn:provesurj} is surjective.

\begin{remark}
One might think that using $\Tits_n(K)/\SL_n \O_K$ instead could give a stronger lower bound, 
but in fact one can show that $\Tits_n(K)/\SL_n\O_K \cong \Tits_n(K)/\GL_n\O_K$.  Since this would not lead to any improvement in the results, we will not prove it, or consider $\Tits_n(K)/\SL_n\O_K$ further.
\end{remark}

It turns out that the space $\Tits_n(K)/\GL_n\O_K$ is also a spherical building 
(see Proposition \ref{proposition:quotientdesc} below).
Proving that the map \eqref{eqn:provesurj} is surjective is difficult because the apartments of $\Tits_n(K)$ that generate $\St_n(K)$ are made up of $n!$ simplices, while the apartments of $\Tits_n(K)/\GL_n \O_K$ have only $2^{n-1}$ simplices.  Since $n!\gg 2^{n-1}$, most apartments of $\Tits_n(K)$ will be draped over a huge number of apartments in the quotient, and it is difficult to describe the resulting homology class. However, we show that by making careful choices, each apartment of $\Tits_n(K)/\GL_n \O_K$ can be lifted to a certain special apartment of $\Tits_n(K)$ for which
all of the remaining $n!-2^{n-1}$ simplices exactly cancel each other out, leaving the desired homology class.  This requires a delicate combinatorial argument.

\subsection{The complex \texorpdfstring{$\PartialBases_n(\O)$}{PBnO} of partial bases} 

The first step in our proof of Theorem~\ref{maintheorem:integralapartments} is to consider a certain ``integral model'' $\PB_n(\O)$ for the Tits building $\Tits_n(K)$. The \emph{complex of partial bases} $\PB_n(\O)$ is the simplicial complex whose maximal simplices (which have dimension $n-1$) are bases of $\O^n$.
Letting $C_{n-1}(\PartialBases_n(\O))$ be the $(n-1)$-chains of $\PartialBases_n(\O)$, we define in \S\ref{section:integralapartments}  an  \emph{integral apartment class map}
\[\phi\colon C_{n-1}(\PartialBases_n(\O))\to \St_n(K)\] 
sending a basis for $\O^n$ to the integral apartment class determined by the corresponding frame of $K^n$.  Theorem~\ref{maintheorem:integralapartments} is the assertion that $\phi$ is surjective.  We prove that $\phi$ is surjective as long as $\PB_n(\O)$ is as highly connected as possible.  Our last main theorem is that this connectivity does indeed hold. 

A $d$-dimensional complex is \emph{$d$-spherical} if it is $(d-1)$-connected, in which case it is homotopy equivalent to a wedge of $d$-spheres.   
 A simplicial complex $X$ is \emph{Cohen--Macaulay} (abbreviated CM) of dimension $d$ if:
\begin{enumerate}[itemsep=2.5pt,parsep=2.5pt,topsep=4pt]
\item $X$ is $d$-spherical, and 
\item for every $k$-simplex $\sigma^k$ of $X$, the link $\Link_X(\sigma^k)$
is $(d-k-1)$-spherical.
\end{enumerate}
\noindent The following theorem is the main technical result of this paper, and is of independent interest.  
\begin{maintheorem}[\textbf{\boldmath$\PartialBases_n(\O_S)$ is Cohen--Macaulay}]
\label{maintheorem:partialbasescm}
Let $\O_S$ be a Dedekind domain of arithmetic type. Assume that $\abs{S}>1$ and that $S$ contains a non-complex place.
Then $\PartialBases_n(\O_S)$ is CM of dimension $n-1$ for all $n \geq 1$.
\end{maintheorem}
Note that, unlike our other results, in Theorem~\ref{maintheorem:partialbasescm} we make no assumption about the class group $\class(\O_S)$.
Theorem~\ref{maintheorem:partialbasescm} applies to any number ring $\O_K$ possessing a real embedding $\O_K\into \R$ (with the exception of $\Z$), as well as to 
every $\O_S$ with $\abs{\O_S^\times}=\infty$ in which some prime $p$ is invertible.
Maazen~\cite{MaazenThesis} proved earlier that for any Euclidean domain $\O$, the complex of partial bases $\PB_n(\O)$ is CM of dimension $n-1$. 

The proof of Theorem~\ref{maintheorem:partialbasescm} is given in \S\ref{section:partialbasescm}.    It is a complicated inductive argument on the rank $n$. To make the induction work, we need to prove that a more general class of  ``complexes of $I$-bases'' relative to an ideal $I\subset \O$ are CM.  It is in proving this stronger statement in the base case $n=2$ that the specific arithmetic hypotheses on $\O_S$ are used. \change

\begin{remark}
\label{remark:vdkintro}
Resolving a conjecture of Quillen, Van der Kallen~\cite[Theorem~2.6]{VanDerKallenStability} 
proved that if a ring $R$ satisfies Bass's stable range condition $SR_d$, then $\PB_n(R)$ is $(n-d)$-connected. 
Any Dedekind domain $\O$ satisfies Bass's stable range condition $SR_3$, so by Van der Kallen's result it was known that $\PB_n(\O)$ 
is $(n-3)$-connected. However, Theorem~\ref{maintheorem:partialbasescm} includes the stronger assertion that $\PB_n(\O)$ 
is $(n-2)$-connected; therefore to prove Theorem~\ref{maintheorem:partialbasescm} requires us to go beyond these known results and show that $\PB_n(\O)$ is as highly connected as if $\O$ were a field or local ring!

Our proof of Theorem~\ref{maintheorem:partialbasescm} uses a stronger condition that can be thought of as ``$SR_{2.5}$'' which Reiner proved holds for all Dedekind domains.  One might hope that Theorem~\ref{maintheorem:partialbasescm} could be proved by an argument mimicking \cite{VanDerKallenStability} by simply substituting this ``$SR_{2.5}$'' condition in place of $SR_3$. But this cannot be so, because the story is more complicated: in Theorem~\ref{theorem:B2notconnected} below we 
show that Theorem~\ref{maintheorem:partialbasescm} is \emph{not true} for all Dedekind domains, nor even those of arithmetic type.  New ideas are needed. See \S\ref{section:expertoverview} for a high-level discussion of the difficulties that arise in extending Van der Kallen's results to this situation, and an explanation of where the hypotheses of Theorem~\ref{maintheorem:partialbasescm} come from.
 \end{remark}


\parasmallskip{Outline of paper}
In \S\ref{section:partialbasescm} we prove  that $\PB_n(\O)$ is CM for appropriate $\O$ (Theorem~\ref{maintheorem:partialbasescm}).   We apply this result in \S\ref{section:integralapartments} to the integral apartment class map to prove the Integrality Theorem (Theorem~\ref{maintheorem:integralapartments}).    This result is applied in \S\ref{section:vanishing} to prove the Vanishing Theorem (Theorem~\ref{maintheorem:vanishing})
for the top degree cohomology of $\SL_n\O_K$ and $\GL_n\O_K$.    The Non-Vanishing Theorem (Theorem~\ref{maintheorem:nonvanishing}) and the  Non-Integrality Theorem (Theorem~\ref{maintheorem:nonintegrality}) are proved in \S\ref{section:prelimnonvanishing}.

\parasmallskip{Acknowledgments} We are grateful to two anonymous referees for their incredibly 
careful readings of this paper, which greatly improved it.   In particular their advice was invaluable for the overview in \S\ref{section:expertoverview}.

\section{\texorpdfstring{$\PartialBases_n(\O_S)$}{Bn(O)} is CM of dimension \texorpdfstring{$n-1$}{n-1}}
\label{section:partialbasescm}

In this section we prove Theorem~\ref{maintheorem:partialbasescm}.  We also prove the following contrasting result.

\begin{theorem}
\label{theorem:B2notconnected}
Let $\O_S$ belong to one of the families of Dedekind domains of arithmetic type:
\begin{enumerate}[topsep=4.5pt,itemsep=4.5pt,parsep=2.5pt]
\item $K=\Q(\sqrt{d})$ and $\O_S=\O_K$ for $d<0$ squarefree with $d\not\in \{-1,-2,-3,-7,-11\}$.
\item $[K: \F_q(T)]<\infty$ with $q$ odd, $S=\{\deg_f\}$, and either $\abs{\class(\O_S)} \neq 1$ or $\deg f\geq 2$.
\end{enumerate}
Then the $(n-1)$-dimensional complex $\PartialBases_n(\O_S)$ is not Cohen--Macaulay for any $n\geq 2$.
\end{theorem}

\begin{remark}
In the second case of the theorem, $\deg_f$ is the discrete valuation determined by a polynomial $f\in \F_q[T]$; in this case one often writes $\O_f$ rather that $\O_{\{\deg_f\}}$.
\end{remark}

In particular, Theorem~\ref{theorem:B2notconnected} implies that $\PB_2(\O_S)$ is not connected, so Van der Kallen's result 
is sharp. For example, neither $\PB_2(\Z[\sqrt{-5}])$ nor $\PB_2(\Z[\frac{1+\sqrt{-43}}{2}])$ is connected, even though both are Dedekind domains and $\Z[\frac{1+\sqrt{-43}}{2}]$ is even a PID.
Similarly, for $K=\F_3(T)$ and $f={T^2+1}$ we have $\O_f\iso \F_3[X,\sqrt{X-X^2}]$, so $\PB_2(\F_3[X,\sqrt{X-X^2}])$ is not connected. 
Theorem~\ref{theorem:B2notconnected} is proved in \S\ref{section:step1}.

\para{Outline}
We begin in \S\ref{section:expertoverview} with a high-level  overview of the proof of Theorem~\ref{maintheorem:partialbasescm}; this overview is separate from the actual proof, and can be skipped if desired, but for both experts and non-experts it should be useful in understanding the structure of the proof. 
In \S\ref{section:basics} we summarize basic properties of Dedekind domains and Cohen--Macaulay complexes.

We prove Theorem~\ref{maintheorem:partialbasescm} by induction, but to do this we must strengthen the inductive hypothesis to a stronger result (Theorem~\ref{thm:BnICM}) applying to a more general family of simplicial complexes; we introduce these in \S\ref{section:partialibases}.
The body of the proof then occupies \S\ref{section:step1}, \S\ref{section:step2}, \S\ref{section:step3}, \S\ref{section:step4}, and \S\ref{section:BnICM}.

\subsection{A bird's-eye view of the proof of Theorem~\ref{maintheorem:partialbasescm}}
\label{section:expertoverview}
In this subsection we give an overview of our approach to Theorem~\ref{maintheorem:partialbasescm}, including a comparison with related classical results and an explanation of the differences that arise in our situation. This overview is written at a higher level than the proof itself, which is given in \S\ref{section:partialibases}, \S\ref{section:step1}, \S\ref{section:step2}, \S\ref{section:step3}, \S\ref{section:step4}, and \S\ref{section:BnICM}.  It may be especially useful for experts who are already familiar with past results on complexes of partial bases, such as the work of Quillen--Wagoner \cite{Wagoner} or Van der Kallen \cite{VanDerKallenStability}. Nevertheless, the overview should be useful for any reader in understanding the big-picture structure of the argument.  We have kept the notation in this section consistent with the remainder of the paper.  For simplicity, some of the definitions in this subsection may be missing or imprecise; precise definitions of all terms used are given 
in the proofs themselves.

\begin{notation}All the connectivity bounds we obtain or desire in this overview are of the same form: if a complex is $(N+1)$-dimensional, our goal is always that it is $N$-connected, or in other words that it is \emph{spherical}. As a result, in this outline there is no need to keep track of the precise connectivity or dimensions that arise; instead, the reader is encouraged to think of ``spherical'' as meaning ``as highly connected as appropriate'', or simply ``nice''. To foster this, by an abuse of notation we will similarly say (in this section only) that the link of a simplex, or a map between complexes, is spherical if it has the appropriate connectivity.\footnote{The reader will have no need of the precise numerics, but for completeness: we say a link of a $(k-1)$-simplex in an $(N+1)$-dimensional complex is spherical if it is $(N-k)$-connected, and a map between $(N+1)$-dimensional complexes is spherical if the map is $N$-connected.} This will allow us to describe the structure of the argument without getting bogged down in numerics.
\end{notation}

\para{Approach to Theorem~\ref{maintheorem:partialbasescm}}
Our main goal is to prove by induction on $n$ that the complex of partial bases $\PB_n(\O)$ is spherical. We will defer the discussion of the base case until later, since certain details would be misleading at this point.

As a stepping-stone, we will make use of the full subcomplex $\PB_n(0)$ on vectors whose last coordinate is either 0 or 1.
The inductive hypothesis implies straightforwardly that $\PB_n(0)$ is spherical (by comparing with the full subcomplex on vectors whose last coordinate is 0, which is isomorphic to $\PB_{n-1}(\O)$). Therefore it suffices to show that the inclusion of the subcomplex $\PB_n(0)\to \PB_n(\O)$ is spherical.

\para{Classical approach}
It is sometimes possible to verify the connectivity of such an inclusion ``simplex-by-simplex'', in the following way. Say that the \emph{0-link} of a partial basis $\{v_1,\ldots,v_k\}$
in $\PB_n(\O)$ is its link \emph{inside} the subcomplex $\PB_n(0)$, i.e.\ the intersection of its link with this subcomplex. If one could verify that the 0-link of each partial basis $\{v_1,\ldots,v_k\}$
in $\PB_n(\O)$  is 
spherical, it would follow that the injection itself is spherical. 

We point out that the \emph{full} link in $\PB_n(\O)$ of a partial basis $\{v_1,\ldots,v_k\}$ is not so hard to understand. The quotient map from $\O^n$ onto $\O^n/\Span{v_1,\ldots,v_k}\iso \O^{n-k}$ induces a map from the link  of $\{v_1,\ldots,v_k\}$ to $\PB_{n-k}(\O)$. This map is not an isomorphism, but it ``fibers'' the former over the latter in a convenient way (cf.\ Def.~\ref{def:fibersCMsplit}), from which one can show that the former is spherical if and only if the latter is. In particular, we know by induction that all links in $\PB_n(\O)$ are spherical; however the question remains of whether the 0-links are spherical.

In the case of a local ring or semi-local ring, where the result was proved by Quillen~\cite[Proposition~1]{Wagoner} and Van der Kallen~\cite[Theorem~2.6]{VanDerKallenStability} respectively, this approach does indeed work: by relating with $\PB_{n-k}(\O)$, one can show that the 0-link of every partial basis is spherical, so the desired connectivity follows. 
However, for Dedekind domains this is simply \textbf{false}. There are two distinct problems that arise.

\para{Problem 1: some partial bases are bad}
The first problem is that even in the simplest cases, there are some partial bases for which there is no hope that their 0-links will be spherical.  Consider for example the vertex $v=(2,5)$ of $\PB_2(\Z)$: to be spherical its 0-link would need to be $(-1)$-connected, i.e.\ nonempty, but it is easy to check that its 0-link is \emph{empty} (the key property is that $2\notin \Z^\times +5\Z$).

To get around this, we single out a special class of partial bases whose 0-link has a better chance of having the appropriate connectivity: we say that a partial basis is \emph{good} if it is contained in some basis where some vector has last coordinate equal to 1 (Def.~\ref{def:good}). We warn the reader that the ``good subcomplex'' formed by these good partial bases is definitely \emph{not} a full subcomplex (Prop.~\ref{prop:charBprime}).

\para{Problem 2: strengthening the inductive hypothesis}
The other problem, even after passing to this good subcomplex,  is that our inductive hypothesis is no longer strong enough. The issue comes from the restriction on the last coordinate.
We saw above that the link of $\{v_1,\ldots,v_k\}$ fibers over $\PB_{n-k}(\O)$. We are now interested in the 0-link, which is the subcomplex including only those vectors whose last coordinate is 0 or 1. This subcomplex no longer fibers over $\PB_{n-k}(\O)$; most simplices in $\PB_{n-k}(\O)$ are inaccessible by such vectors. Indeed, if $J\subset \O$ is the ideal generated by the last coordinates of $v_1,\ldots,v_k$, the only vectors in $\PB_{n-k}(\O)$ in the image of the 0-link are those whose last coordinate is \emph{congruent} to $0$ or $1$ modulo $J$.

To deal with this, we introduce a new complex $\PB_m(I)$ defined relative to an ideal $I\subset \O$, consisting only of vectors whose last coordinate is $\equiv 0$ or $1\bmod{I}$ (with some additional conditions, see Def.~\ref{def:Ibasis}).
To continue the induction, we are forced to go back and strengthen our inductive hypothesis to the stronger claim (Thm~\ref{thm:BnICM}) that $\PB_m(I)$ is spherical for \emph{every} ideal $I\subset \O$.

In most cases this resolves the issue: the 0-link in $\PB_n(\O)$ fibers over $\PB_{n-k}(J)$, so the strengthened inductive hypothesis gives us the necessary connectivity.

However there is a further problem in one case. If $J=\O$ (a case that one might have expected to be \emph{easier}), although the 0-link does map to $\PB_{n-k}(\O)$, it is \emph{not} fibered in the same convenient way. To handle this, we are forced to introduce a weaker notion of ``fibered relative to a core'' (Def.~\ref{def:fibersCMsplit}); although this  condition is quite technical, it seems to be necessary to push the induction through in this case. We point out that this step is one of the main reasons that in Thm \ref{maintheorem:partialbasescm}/\ref{thm:BnICM} we have to prove that $\PB_n(\O)$ is Cohen--Macaulay, and not just highly connected (cf. Remark~\ref{remark:trouble}).

\para{Returning to the base case} Let us return to the discussion of the base case, which we set aside earlier. The key base case for us is  $n=2$, in which case Thm~\ref{thm:BnICM} states that the graph $\PB_2(I)$ should be \emph{connected} for any ideal $I$. The connectivity of this graph is directly related to the question of elementary generation for congruence subgroups of $\SL_2(\O)$ relative to the ideal $I$ (Prop.~\ref{prop:basecasen2}). Therefore we can use the work of Vaserstein--Liehl and Bass--Milnor-Serre on relative elementary generation to show that $\PB_2(I)$ is indeed connected; our conditions on $\O$ come directly from the hypotheses of Vaserstein--Liehl and Bass--Milnor-Serre.

We can now see why we deferred the base case until the end: if we had discussed it before introducing the complexes $\PB_2(I)$ and Thm~\ref{thm:BnICM}, it would have seemed that  the base case of Thm~\ref{maintheorem:partialbasescm} needs only the connectivity of $\PB_2(\O)$. But the connectivity of $\PB_2(\O)$ only needs  elementary generation for $\SL_2(\O)$ itself, rather than for its congruence subgroups; the former is a much easier condition, and holds for a broader class of Dedekind domains. In particular, one could not see from this perspective where the restrictions on $\O$ in Thm~\ref{maintheorem:partialbasescm}  come from.

\smallskip
The $n=2$ base case also plays another completely separate role in our proof: it turns out to be  key in the \emph{first} part of the inductive step, letting us show that we really can reduce to the ``good subcomplex'' discussed above. This is also where we use the ``$SR_{2.5}$'' condition mentioned in Remark~\ref{remark:vdkintro}.

\para{Structure of the proof} \S\ref{section:basics} summarizes elementary facts about Dedekind domains and Cohen--Macaulay complexes; experts can skip this section. 
In \S\ref{section:partialibases} we introduce the complexes $\PB_n(I)$ and state Theorem~\ref{thm:BnICM}, which includes Theorem~\ref{maintheorem:partialbasescm} as a special case.

The remainder of \S\ref{section:partialbasescm} is occupied by the proof of Theorem~\ref{thm:BnICM}, by induction on $n$.
In \S\ref{section:step1} we establish
the base cases $n \leq 2$.
In \S\ref{section:step2} we define a subcomplex $\PartialBasesGood_n(I) \subset \PartialBases_n(I)$
consisting of ``good simplices'' and prove that
$(\PB_n(I),\PBg_n(I))$ is $(n-2)$-connected.
In \S\ref{section:step3} we prove that $(\PBg_n(I),\PB_n(0))$ is $(n-2)$-connected.
In \S\ref{section:step4} we prove that $\PB_n(0)$ is $(n-2)$-connected.
In \S\ref{section:BnICM} we prove the connectivity of links, and
assemble these pieces to prove Theorem~\ref{thm:BnICM}.

\subsection{Dedekind domains and Cohen--Macaulay complexes}
\label{section:basics}

\subsubsection{Dedekind domains}

Let $\O$ be a Dedekind domain with field of fractions $K$. We recall some basic facts
about $\O$; for proofs see e.g.\ \cite{MilnorKtheory}.

\para{Projective modules}
Every submodule of a projective $\O$-module is 
projective. 
If $M$ is a finitely generated projective $\O$-module, the \emph{rank} of $M$ is defined by 
$\Rank(M)\coloneq \dim_K(M \otimes_\O K)$.  The following standard lemma summarizes
some properties of projective $\O$-modules.

\begin{lemma}
\label{lemma:summandsubspace}
Let $\O$ be a Dedekind domain with field of fractions $K$ and let $M$ be a finitely generated
projective $\O$-module of rank $n \geq 1$.
\begin{enumerate}[label={\normalfont (\alph*)},topsep=2pt,itemsep=1pt,parsep=1pt]
\item A submodule $U\subset M$ is a direct summand of $M$ if and only if $M/U$ is torsion-free.
\item If $U$ and $U'$ are summands of $M$ and $U\subset U'$, then $U$ is a summand of $U'$.
\item The assignment $U\mapsto U\tensor_\O K$ defines a bijection
\[\{\text{Direct summands of $M$}\} \longleftrightarrow \{\text{$K$-subspaces of $M\tensor_\O K$}\}\]
with inverse $V\mapsto V\cap M$.
\end{enumerate}
\end{lemma}

\paranoskip{The class group}
The \emph{class group} $\class(\O)$ is the
set of isomorphism classes of rank $1$ projective $\O$-modules.  The tensor product endows $\class(\O)$ with the structure 
of an abelian group, which we write additively; the identity element is given by the free module $\O$.
For every rank $n$ projective $\O$-module $M$ there is a unique rank~1 projective $\O$-module $I\in \cl(\O)$ such 
that $M\iso \O^{n-1}\oplus I$.  We write $[M] = I\in \cl(\O)$ for this element.  For two finite rank 
projective $\O$-modules $M$ and $M'$, we have the identity $[M \oplus M'] = [M] + [M']$.  Two such 
modules $M$ and $M'$ are isomorphic if and only if $\Rank(M)=\Rank(M')$ and $[M]=[M']$.

\para{The complex of partial bases}
A \emph{partial basis} for $\O^n$ is a set $\{v_1,\ldots,v_k\}$ of elements of $\O^n$ that can be completed to a basis
$\{v_1,\ldots,v_k,v_{k+1},\ldots,v_n\}$ for $\O^n$.
As we mentioned in the introduction, the \emph{complex of partial bases} 
$\PartialBases_n(\O)$ is the simplicial
complex whose $(k-1)$-simplices are partial bases $\{v_1,\ldots,v_k\}$ for $\O^n$.
Steinitz proved in 1911 that when $\O$ is a Dedekind domain, a vector $v=(a_1,\ldots,a_n) \in \O^n$ lies 
in a basis for $\O^n$ if and only if $v$ is \emph{unimodular}, that is, the  coordinates $\{a_1,\ldots,a_n\}$ generate the 
unit ideal of $\O$; see \cite{ReinerUnimodular}.  The vertices of $\PB_n(\O)$ are therefore
the unimodular vectors in $\O^n$.

\begin{lemma}
Let $\O$ be a Dedekind domain.
\label{lemma:partialbasissub}
\begin{enumerate}[label={\normalfont (\alph*)},topsep=2pt,itemsep=1pt,parsep=1pt]
\item A subset $\{v_1,\ldots,v_k\}\subset \O^n$ is a partial basis of $\O^n$ if and only if $\spn(v_1,\ldots,v_k)$ 
is a rank $k$ direct summand of $\O^n$.
\item Let $M$ be a free summand of $\O^n$. A subset $\{v_1,\ldots,v_k\}\subset M$ is a partial basis of $M$ 
if and only if $\{v_1,\ldots,v_k\}$ is a partial basis of $\O^n$. 
\end{enumerate}
\end{lemma}
\begin{proof}
Set $U=\spn(v_1,\ldots,v_k)$.  If $\{v_1,\ldots,v_k\}$ is a partial basis for $\O^n$ then
$U$ is  a rank $k$ direct summand of $\O^n$.  We must prove the converse, so assume that $U$ is a
rank $k$ direct summand of $\O^n$.  Since $U$ is a rank $k$ projective module generated by $k$ elements,
it must be free.  Write $\O^n = U \oplus W$.  Since $[W]=[\O^n]-[U]=0-0=0$, we have $W\iso \O^{n-k}$.  Adjoining a basis 
for $W$ then yields a basis for $\O^n$, so $\{v_1,\ldots,v_k\}$ is a partial basis and (a) follows.
Part (b) follows from the characterization in part (a) together with Lemma~\ref{lemma:summandsubspace}(b).
\end{proof}

\subsubsection{Cohen--Macaulay complexes}
A space is {\em $(-1)$-connected} if and only if it is nonempty; every space is $\ell$-connected for $\ell \leq -2$.
A simplicial complex $X$ is {\em $d$-dimensional} if it contains a $d$-simplex but no $(d+1)$-simplices.  It is {\em $d$-spherical}
if it is $d$-dimensional and $(d-1)$-connected.
Finally, $X$ is Cohen--Macaulay (CM) of dimension $d$ if $X$ is $d$-spherical
and for every $k$-simplex $\sigma^k$ of $X$, the link $\Link_X(\sigma^k)$ is $(d-k-1)$-spherical.

\para{Low dimensions}
A simplicial complex $X$ is CM of dimension 0 if $X$ is 0-dimensional.   $X$ is CM of dimension 1 if $X$ is 1-dimensional and connected.

\para{Links are CM}
If $X$ is CM of dimension $d$, then $\Link_X(\sigma^k)$ is CM of dimension $d-k-1$ for all $\sigma^k\in X$ (where the definition requires only $(d-k-1)$-spherical); see \cite[Proposition~8.6]{QuillenPoset}.  We remark that \cite[Proposition~8.6]{QuillenPoset} concerns CM posets
rather than CM simplicial complexes; the desired result for simplicial complexes is
obtained by considering the poset of simplices of the simplicial complex.  Here we
are using the fact that a simplicial complex is CM if and only if its poset of simplices
is, which follows from \cite[(8.5)]{QuillenPoset}. 

\para{A sufficient condition to be CM}
If a $d$-dimensional simplicial complex $X$ has the property that every simplex is contained in a $d$-simplex, then $\Link_X(\sigma^k)$ is $(d-k-1)$-dimensional for all $\sigma^k\in X$.  Therefore, to prove that $X$ is CM of dimension $d$, it suffices to show that $X$ is $(d-1)$-connected, and that for every $k$-simplex $\sigma^k\in X$ the link $\Link_X(\sigma^k)$ is $(d-k-2)$-connected. 

\subsection{The complex \texorpdfstring{$\PartialBases_n(I)$}{PBnI} of partial \texorpdfstring{$I$-bases}{I-bases}}
\label{section:partialibases}

One of the key insights in our proof of Theorem~\ref{maintheorem:partialbasescm} is
that to induct on $n$, one must strengthen the inductive hypothesis by 
working with a more general family of complexes.

\begin{definition}[{\bf Partial \boldmath$I$-basis}]
\label{def:Ibasis}
Fix a surjection $L\colon \O^n \onto \O$.  Given an ideal $I \subset \O$, a basis $\{v_1,\ldots,v_n\}$
for $\O^n$ is called  an \emph{$I$-basis} if \[L(v_i)\equiv 0\bmod{I}\text{ or }L(v_i)\equiv 1\bmod{I}\qquad \text{for all }1 \leq i \leq n.\]
A  partial basis is called a \emph{partial $I$-basis} if it can be completed to an
$I$-basis.  The \emph{complex of partial $I$-bases}, denoted $\PartialBases_n(I)$,
is the complex whose $(k-1)$-simplices are partial $I$-bases $\{v_1,\ldots,v_k\}$.
\end{definition}

\begin{remark}
\label{remark:nonuniqueBnI}
The definition of $\PartialBases_n(I)$ depends on the choice of $L$.  However, if
$L'\colon\O^n \onto \O$ is a different surjection, then both $\ker L$ and $\ker L'$ are projective summands with $\Rank(\ker L)=\Rank(\ker L')=n-1$ and $[\ker L]=[\ker L']=[\O^n]-[\O]=0\in \cl(\O)$. Therefore there exists an isomorphism $\Psi\colon \O^n \to \O^n$
such that $L = L' \circ \Psi$. The resulting isomorphism $\PB_n(\O)\xrightarrow{\iso}\PB_n(\O)$ induces an isomorphism between the complexes
$\PartialBases_n(I)$ defined with respect to $L$ and $L'$.  Because of this, we will not concern ourselves
with the choice of $L$ except when it is necessary to clarify our proofs.
\end{remark}

We record some observations regarding $I$-bases and partial $I$-bases.
\begin{enumerate}[(a),topsep=2pt,itemsep=2pt,parsep=2pt]
\item 
When $I = \O$, the condition of Definition~\ref{def:Ibasis} is vacuous, so $\PartialBases_n(I) = \PartialBases_n(\O)$.
\item If $I\subset J$ then any $I$-basis is a $J$-basis by definition.  Therefore any 
partial $I$-basis is a partial $J$-basis, so $\PartialBases_n(I)$ is a subcomplex of $\PartialBases_n(J)$. In particular, $\PB_n(I)$ is a subcomplex of $\PB_n(\O)$.
\item We allow the case $I=(0)$, in which case we simplify notation by writing $\PartialBases_n(0)$ 
for $\PartialBases_n(I)$.  The subcomplex $\PartialBases_n(0)$ is contained in $\PartialBases_n(I)$ for any ideal 
$I$ since $(0)\subset I$.  We will see in the proof of Proposition~\ref{prop:BprimeB0} below that
$\PartialBases_n(0)$ is the full subcomplex\footnote{Recall that a subcomplex $L$ of a simplicial complex $K$ is {\em full} if the simplices of $L$ are precisely those simplices of $\sigma\subset K$ whose vertices lie in $L$; in particular, a full subcomplex $L$ is determined by its vertices.} of $\PartialBases_n(\O)$ on the unimodular vectors 
$v\in \O^n$ which satisfy $L(v)=0$ or $L(v)=1$.
\change 
\item A partial basis $\{v_1,\ldots,v_k\}$ may not be a 
partial $I$-basis even if $L(v_i)$ is congruent to $0$ or $1$ modulo $I$ for all $1 \leq i \leq k$.
For example, let $\O=\Z$, let $I=5\Z$, and take $L\colon\Z^2\onto\Z$ to be $L(b,c)=c$. 
The vector $v\coloneq (2,5)\in \Z^2$ determines a partial basis $\{v\}$ since 
$\{(2,5),(1,3)\}$ is a basis for $\Z^2$.  However, although $L(v)=5\equiv 0\bmod I$, 
the partial basis $\{v\}$ is not a partial $I$-basis.  Indeed, any basis $\{(2,5),(b,c)\}$ 
must have $2c-5b=\pm1$, which implies that $c\equiv 2\text{ or }3\bmod{5\Z}$.  Therefore 
$v$ is not contained in any $I$-basis for $\Z^2$, so $\{v\}$ is not a vertex of $\PB_2(I)$.
\item A consequence of the previous paragraph is that  for $n\geq 3$ the complex of partial $I$-bases $\PartialBases_n(I)$ is \emph{not} a full subcomplex of 
$\PartialBases_n(\O)$ in general. With $\O=\Z$ and $I=5\Z$ as before, take $L\colon\Z^3\onto\Z$ to be $L(a,b,c)=c$. Using the additional room in $\Z^3$, the vector $v\coloneq (0,2,5)$ can be extended to an $I$-basis such as $\{(0,2,5),(3,1,0),(-1,0,1)\}$.
The same is true of the vector $w\coloneq (1,2,5)$ so both $v$ and $w$ belong to $\PB_3(I)$. Since $\{v,w,(0,1,3)\}$ is a basis, the vertices $v$ and $w$ determine an edge of $\PB_3(\O)$. But they do not form a partial $I$-basis, because as before any basis $\{(0,2,5),(1,2,5),(a,b,c)\}$ must have $2c-5b=\pm 1$. Thus this edge of $\PB_3(\O)$ is absent from $\PB_3(I)$, even though its endpoints do belong to $\PB_3(I)$.
\item In the previous two remarks, one might be concerned that $\Z$ does not satisfy the hypotheses of Theorem~\ref{maintheorem:partialbasescm}. But in both remarks we can obtain an example in $\O=\Z[\frac{1}{19}]$, which does satisfy  the hypotheses of Theorem~\ref{maintheorem:partialbasescm}, by taking the same vectors; the necessary property is simply that $2\notin \O^\times +5\O$, which still holds when $\O=\Z[\frac{1}{19}]$.
\end{enumerate}

\medskip
Our main result concerning the complexes $\PartialBases_n(I)$ is as follows.
\begin{theorem-prime}{maintheorem:partialbasescm}
\label{thm:BnICM}
 Let $\O_S$ be a Dedekind domain of arithmetic type.  Assume that $\abs{S}>1$ and $S$ contains a non-complex place.  Then $\PartialBases_n(I)$ is CM of dimension $n-1$ for any ideal $I \subset \O_S$.
\end{theorem-prime}

Theorem~\ref{maintheorem:partialbasescm} is the special case $I=\O_S$ of Theorem~\ref{thm:BnICM}. 

\subsection{Step 1: Base case, \texorpdfstring{$n\leq 2$}{n <= 2}}
\label{section:step1}

Our inductive proof of Theorem~\ref{thm:BnICM} has two base cases.  The first is easy.

\begin{proposition}
\label{prop:basecasen1}
For any Dedekind domain $\O$ and any ideal $I\subset \O$, the complex $\PB_1(I)$ is CM of dimension 0.
\end{proposition}
\begin{proof}
A surjection $L\colon \O^1\onto \O$ must be an isomorphism, so there exists a unique $v\in \O^1$ with $L(v)=1$. Since $\{v\}$ is an $I$-basis of $\O^1$ for any $I$, the complex $\PB_1(I)$ contains the 0-simplex $\{v\}$, and thus is 0-dimensional as desired.
\end{proof}

The second base case is more involved.   It is the key place where the arithmetic hypotheses on $\O_S$ enter into the proof of Theorem~\ref{thm:BnICM}.

\begin{proposition}
\label{prop:basecasen2}
Let $\O=\O_S$ be a Dedekind domain of arithmetic type. Assume that $\abs{S}>1$ and that $S$ does not consist solely of complex places. Then for any ideal $I\subset \O$, the complex $\PB_2(I)$ is CM of dimension 1.
\end{proposition}
\begin{remark}
Our proof of Proposition~\ref{prop:basecasen2} could be shortened slightly if we were only interested in Proposition~\ref{prop:basecasen2}; we go a bit further here to simplify the proof of Theorem~\ref{theorem:B2notconnected} afterwards.
\end{remark}
\change
\begin{proof}[Proof of Proposition~\ref{prop:basecasen2}]
The complex $\PB_2(I)$ is $1$-dimensional, so it is enough to show that $\PB_2(I)$ is connected.
Since $L$ is surjective we can fix a basis $\{e_1,e_2\}$ for $\O^2$ such
that $L(e_1)=0$ and $L(e_2) = 1$. 
This is an $I$-basis, so it defines an edge $e = \{e_1,e_2\}$
of $\PartialBases_2(I)$.

Let $\Gamma_1(\O,I)$ be the group
\[\Gamma_1(\O,I) \coloneq \big\{\left(\begin{smallmatrix} a&b\\c&d \end{smallmatrix}\right) \in \SL_2 \O\,\big|\,\left(\begin{smallmatrix} a&b\\c&d \end{smallmatrix}\right) \equiv \left(\begin{smallmatrix} \ast&\ast\\0&1 \end{smallmatrix}\right)\bmod{I}\ \big\}.\]
Letting $\SL_2 \O$ act on $\O^2$ via the basis
$\{e_1,e_2\}$, the subgroup $\Gamma_1(\O,I)$ consists of the automorphisms preserving the composition $\O^2\overset{L}{\to}\O\onto \O/I$. \change 
 In particular, the $\Gamma_1(\O,I)$-orbit of $e$ consists of $I$-bases and thus is contained in $\PartialBases_2(I)$.  Let $Z\subset \PB_2(I)$ be the subgraph determined by this orbit.

If $I\neq \O$ then the action of $\Gamma_1(\O,I)$ on $Z$ is without inversions.  Indeed,
in this case vertices with $L(v)\equiv 0\bmod{I}$ cannot be exchanged with those with $L(v)\equiv 1\bmod{I}$.  
The edge $e$ is a fundamental domain for the action of $\Gamma_1(\O,I)$ on $Z$.  The $\Gamma_1(\O,I)$-stabilizers
of $e_1$ and $e_2$ are exactly the elementary matrices that lie
in $\Gamma_1(\O,I)$, so by a standard argument in geometric group theory
(see, e.g., \cite[Lemma I.4.1.2]{SerreTrees}) the complex $Z$ is connected if and only if the group $\Gamma_1(\O,I)$
is generated by elementary matrices. This argument does not apply verbatim when $I=\O$,
but it is easy to verify nevertheless that $Z$ is connected if and only if $\SL_2\O$ is generated by elementary matrices.

Vaserstein (see the Main Theorem of \cite{VasersteinGen}, and  Liehl~\cite{Liehl} for a correction) 
proved that if $\O=\O_S$ is a Dedekind domain of arithmetic type with $\abs{S}>1$, then for any ideal $I\subset \O$ the subgroup $E$ of $\Gamma_1(\O,I)$
generated by elementary matrices is normal, and the quotient
$\Gamma_1(\O,I)/E$ is isomorphic to the relative $K$-group $\text{SK}_1(\O,I)$.  We remark that
Vaserstein's result does not apply to $I=(0)$, but in this case $\Gamma_1(\O,I)$ consists solely of elementary matrices.
In their resolution of the Congruence Subgroup Problem, Bass--Milnor--Serre~\cite[Theorem~3.6]{BassMilnorSerre} proved that if $\O$ is a Dedekind domain of arithmetic type that is \emph{not} totally imaginary, then $\text{SK}_1(\O,I) = 0$ for all ideals $I\subset \O$. Our assumptions guarantee that both of these results apply, so $\Gamma_1(\O,I)$ is generated by elementary matrices and $Z$ is connected.

It is thus enough to show that the pair $(\PB_2(I),Z)$ is $0$-connected, i.e.\ that every vertex
$v\in \PB_2(I)$ can be connected by a path to a vertex in $Z$. By definition, any vertex $v\in\PB_2(I)$ 
is contained in an $I$-basis $\{v,w\}$ for $\O^2$.
We first show that if $L(v)\equiv 1 \bmod{I}$ then $v$ lies in $Z$.  

By possibly replacing $w$ by $w-v$, we can assume that $L(w)\equiv 0\bmod{I}$. There exists a unique $g\in \GL_2 \O$ such that $g\cdot e_1=w$ and $g\cdot e_2=v$. Let $u = (\det g)^{-1} w$.  The element $\gamma$ satisfying $\gamma\cdot e_1=u$ and $\gamma\cdot e_2=v$ lies in $\SL_2 \O$. Since $L(u)\equiv 0\bmod{I}$ and $L(v)\equiv 1\bmod{I}$, we see that $\gamma\in \Gamma_1(\O,I)$. Therefore $v$ lies in the $\Gamma_1(\O,I)$-orbit of $e_2$ and $v\in Z$.

If instead $L(v)\not\equiv 1\bmod{I}$, then since $\{v,w\}$ is an $I$-basis we must have $L(w)\equiv 1\bmod{I}$, so $w$ lies in $Z$. We conclude that every vertex of $\PB_2(I)$ is connected to $Z$ by a path of length at most $1$, as desired.
\end{proof}

Before moving on to Step 2, we take a moment to prove Theorem~\ref{theorem:B2notconnected}.
\begin{proof}[Proof of Theorem~\ref{theorem:B2notconnected}]
We first prove the theorem in the case $n=2$. 
We can break the hypotheses of Theorem~\ref{theorem:B2notconnected} into the following three cases, and Cohn proved in each that $\GL_2 \O_S$ is not generated by elementary matrices together with diagonal matrices: for the quadratic imaginary case see \cite[Theorem 6.1]{Cohn}, for the function field case with $\abs{\class(\O_S)}\neq 1$ see \cite[Corollary~
5.6]{Cohn}, and for the function field case with $\deg f\geq 2$ see \cite[Theorem 6.2]{Cohn}. Thus certainly $\SL_2\O_S$ cannot be generated by elementary matrices. We saw in the proof of Proposition~\ref{prop:basecasen2} that the 1-dimensional complex $\PB_2(\O_S)$ is connected if and only if $\SL_2 \O_S$ is generated by elementary matrices. Therefore $\PB_2(\O_S)$ is disconnected and not Cohen--Macaulay.

Now assume that $n\geq 3$. Choose a basis $\{e_1,\ldots,e_n\}$ for $\O^n$; let $\sigma^{n-3}=\{e_1,\ldots,e_{n-2}\}$ and $W=\langle e_{n-1},e_n\rangle$. Let $\PB_W(\O_S)$ denote the complex of partial bases of $W$; since $W\iso \O_S^2$ there is an isomorphism $\PB_W(\O_S)\iso \PB_2(\O_S)$. By the previous paragraph, $\PB_W(\O_S)$ is disconnected.

Consider the subcomplex $L\coloneq \Link_{\PB_n(\O_S)}(\sigma^{n-3})$. If $\PB_n(\O_S)$ were CM of dimension $n-1$ then $L$ would be $(n-1)-(n-3)-1=1$-spherical and therefore connected. 
Since $\{w_1,w_2\}$ is a basis of $W$ if and only if $\{e_1,\ldots,e_{n-2},w_1,w_2\}$ is a basis of $\O^n$, there
is a natural inclusion $i\colon \PB_W(\O_S)\into L$.  Let $\pi\colon \O^n\onto W$ be the projection with kernel $\langle e_1,\ldots,e_{n-2}\rangle$. Since $\{e_1,\ldots,e_{n-2},v_1,v_2\}$ is a basis of $\O^n$ if and only if $\{\pi(v_1),\pi(v_2)\}$ is a basis of $W$, the projection $\pi$ induces a simplicial map $\pi\colon L\to \PB_W(\O_S)$. By definition $i\colon \PB_W(\O_S)\into L$ is a section of this projection, so $\pi\colon L\to \PB_W(\O_S)$ is surjective. Since $\PB_W(\O_S)$ is disconnected, this shows that $L=\Link_{\PB_n(\O_S)}(\sigma^{n-3})$ is disconnected, so the $(n-1)$-dimensional complex $\PB_n(\O_S)$ is not Cohen--Macaulay.
\end{proof}

\begin{remark}
Other examples can be deduced from Geller~\cite{Geller}; for example $\PB_n(\O_f)$ is not Cohen--Macaulay for $\O_f\iso \F_2[X,Y]/(X^2+XY+Y^2+X)$, which occurs as $\O_f\subset \F_2(T)$ for $f=T^2+T+1$. In contrast, Vaserstein~\cite{VasersteinGen} proved that when $\abs{S}\geq 2$ the group $\SL_2\O_S$ is generated by elementary matrices, so in that case $\PB_2(\O_S)$ is indeed connected.
\end{remark}

\subsection{Step 2: \texorpdfstring{$(\PB_n(I),\PBg_n(I))$}{(Bn(I),Bn-good(I))} is \texorpdfstring{$(n-2)$-connected}{(n-2)-connected}}
\label{section:step2}

We begin by defining certain bases to be ``good bases''.

\begin{definition}
\label{def:good}
We say that a basis $\{v_1,\ldots,v_n\}$ for $\O^n$ is \emph{good} if $L(v_i)=1$ for at least one $1\leq i\leq n$. We say that $\{v_1,\ldots,v_k\}$ is a \emph{partial good $I$-basis} if it is contained in a good $I$-basis. We denote by $\PBg_n(I)$ the complex of partial good $I$-bases.
\end{definition}

The following proposition is the main result of this section.

\begin{proposition}
\label{prop:BBprime}
Let $\O$ be a Dedekind domain and let $I\subset \O$ an ideal. Assume that $\PB_2(I)$ is CM of dimension 1. Then for any $n\geq 3$ the pair $(\PB_n(I),\PBg_n(I))$ is $(n-2)$-connected.
\end{proposition}

\subsubsection{The link of an \texorpdfstring{$(n-2)$-simplex}{(n-2)-simplex}}
We need a few preliminary results before starting the proof of 
Proposition~\ref{prop:BBprime}. The first ingredient in the proof of Proposition~\ref{prop:BBprime} is the following lemma.
\begin{definition}
Given a simplex $\sigma = \{v_1,\ldots,v_k\}$ of $\PartialBases_n(\O)$, define $V_{\sigma}$ to be the direct summand of $\O^n$ with basis $\{v_1,\ldots,v_k\}$, and define $I_{\sigma}$ to be the ideal $L(V_{\sigma}) \subset \O$.
\end{definition}
\begin{lemma}
\label{lemma:complement2}
Let $n \geq 3$. Consider an $(n-2)$-simplex $\sigma^{n-2}\in\PartialBases_n(I)$, and let $v$ be a vertex of $\sigma$ with $L(v)\neq 0$.  If $I_\sigma \not\subset I$,
 assume additionally that $L(v) \equiv 1\bmod{I}$. Then there exists a vertex $w$ of $\Link_{\PartialBases_n(I)}(\sigma)$ such that
$I_{\{v,w\}} = \O$.
\end{lemma}

The proof of Lemma \ref{lemma:complement2} requires the following result of Reiner, which
is a strengthening for Dedekind domains of Bass's stable range condition $\text{SR}_3$.

\begin{lemma}[Reiner]
\label{lemma:reiner}
Let $\O$ be a Dedekind domain and let $\{b_1,\ldots,b_n\}$ be
elements of $\O$ that generate the unit ideal.  Assume that $n \geq 3$ and $b_1 \neq 0$.
Then there exists $c_3,\ldots,c_n \in \O$ such that $b_1$ and
$b_2 + c_3 b_3 + \cdots + c_n b_n$ generate the unit ideal.
\end{lemma}
\begin{proof}
In \cite[bottom of p. 246]{ReinerUnimodular}, Reiner proved that there exists some
$d \in \O$ such that the elements $\{b_1,\ldots,b_{n-2},b_{n-1}+d b_n\}$ of $\O$
generate the unit ideal.  We remark that Reiner does not explictly require $b_1 \neq 0$,
but this is needed for his proof.  Repeatedly applying this, the lemma follows.
\end{proof}

\begin{proof}[Proof of Lemma \ref{lemma:complement2}]
Since $\sigma^{n-2}\in \PartialBases_n(I)$ is a partial $I$-basis, there exists some $z \in \O^n$ such that
$\sigma \cup \{z\}$ is an $I$-basis.  Let $b_1 = L(v)$ and $b_2 = L(z)$; by hypothesis $b_1\neq 0$. Write $\sigma^{n-2} = \{v,v_3,\ldots,v_n\}$ and let $b_i = L(v_i)$ for $3 \leq i \leq n$.
Since $\sigma \cup \{z\}$ is a basis for $\O^n$, the set $\{b_1,\ldots,b_n\} \subset \O$ generates
the unit ideal.  Since $n \geq 3$ and $b_1 \neq 0$ and $\{b_1,\ldots,b_n\}$ generates the unit ideal, Lemma \ref{lemma:reiner} implies that 
there exists $c_3,\ldots,c_n \in \O$ such that
the two elements $b_1$ and $\xi\coloneq b_2 + c_3 b_3+c_4 b_4+\cdots+c_n b_n$ generate the unit ideal.
Set $y= z + c_3 v_3 + \cdots +c_n v_n$. The ideal $I_{\{v,y\}}$ is generated by $L(v) = b_1$ and $L(y) = \xi$, so by construction $I_{\{v,y\}}=\O$. Since $\sigma \cup \{z\}$ is a basis for $\O^n$, so is $\sigma\cup \{y\}$.

If $I_\sigma \subset I$ then $L(v_3)\equiv\cdots\equiv L(v_n)\equiv 0\bmod{I}$, so $L(y)\equiv L(z)\bmod{I}$.  Since
$z\in \PartialBases_n(I)$, we know that $L(z)$ is congruent to $0$ or $1$ modulo $I$. Therefore when
$I_\sigma \subset I$, the basis $\sigma\cup\{y\}$ is an $I$-basis.  We conclude that $w \coloneq y\in \Link_{\PB_n(I)}(\sigma)$ satisfies the conditions of the lemma.

When $I_\sigma\not\subset I$, the basis $\sigma \cup \{y\}$ may not be an $I$-basis since $L(y)=\xi$ may not be congruent to $0$
or $1$ modulo $I$.  In this case, by hypothesis $L(v)\equiv 1\bmod{I}$, so the vector $w\coloneq y-\xi v$ satisfies
\[L(w)=L(y)-\xi\cdot L(v)\equiv \xi-\xi=0\bmod{I}.\]
Therefore the basis $\sigma \cup \{w\}$ is an $I$-basis, so $w\in \Link_{\PartialBases_n(I)}(\sigma)$. Since
\[I_{\{v,w\}}=(L(v),L(w))=(L(v),L(y)-\xi\cdot L(v))=(L(v),L(y))=(b_1,\xi)=\O,\]
the vector $w$ satisfies the conditions of the lemma.
\end{proof}

\subsubsection{Characterizing good partial bases}
The second ingredient in the proof of Proposition~\ref{prop:BBprime} is the following characterization of good partial bases.

\begin{proposition}
\label{prop:charBprime}
Let $\sigma^k=\{v_1,\ldots,v_{k+1}\}$ be a $k$-simplex of $\PB_n(\O)$ \change with $k<n-1$.
\begin{enumerate}[label={\normalfont (\arabic*)},topsep=2pt,itemsep=2pt,parsep=2pt]
\item Assume that \change $L(v_i)\equiv 0\text{ or }1\bmod{I}$ for $i=1,\ldots,k+1$.  Then the following are equivalent.
\begin{enumerate}[label={\normalfont (\roman*)},topsep=1pt,noitemsep]
\item $\{v_1,\ldots,v_{k+1}\}$ is contained in a good basis.
\item $\{v_1,\ldots,v_{k+1}\}$ is contained in a good $I$-basis $\{v_1,\ldots,v_n\}$ s.t.\ $L(v_i)=1$ for $i\geq k+2$.
\item There exists a complement $W$ such that $\O^n=V_\sigma\oplus W$ and $L(W)=\O$.
\end{enumerate}
\item If $I_\sigma=(0)$ or $I_\sigma=\O$ or $k<n-2$, then $\sigma$ is contained in a good basis.
\end{enumerate}
\end{proposition}

Before proving Proposition~\ref{prop:charBprime}, we introduce some terminology.
We say that $\sigma$ is a \emph{good simplex} if the three conditions of Proposition~\ref{prop:charBprime}(1) are satisfied, and call the subspace $W$ of condition (iii) a \emph{good complement} for $\sigma$.
In this case we denote by $\PB_W(I)$ the complex of partial $I$-bases for $W$, whose simplices are subsets of $W$ that can be extended to an $I$-basis for $W$ (with respect to $L$). Since $[W]=[\O^n]-[V_\sigma]=0$, any good complement is a free $\O$-module, and $L|_W\colon W\onto \O$ is surjective by definition. Thus by Remark~\ref{remark:nonuniqueBnI}, there exists an isomorphism $\PB_W(I)\iso \PB_{\ell}(I)$, where $\ell=\Rank(W)$.

\begin{proof}[{Proof of Proposition~\ref{prop:charBprime}}]
To prove the first part of the proposition we will show that\linebreak (i) $\implies$ (ii) $\implies$ (iii) $\implies$ (i).

Assume (i), so $\sigma=\{v_1,\ldots,v_{k+1}\}$ is contained in a good basis 
\[\{v_1,\ldots,v_{k+1},u_{k+2},\ldots,u_n\}.\]
 
By definition, some vector $v$ lying in this good basis satisfies $L(v)=1$.
Define $v_i=u_i+(1-L(u_i))v$ for $k+2 \leq i \leq n$.  The set $\{v_1,\ldots,v_n\}$ is then a basis
satisfying $L(v_i)=1$ for $k+2 \leq i \leq n$. 
Since $k+1<n$, this is a good basis, and since $L(v_i)\equiv 0\text{ or }1\bmod{I}$ for $1 \leq i \leq k+1$ by assumption, it is also an $I$-basis. 
Therefore (i) $\implies$ (ii).  

Assume (ii), and define $W=\langle v_{k+2},\ldots,v_n\rangle$.  We certainly have $\O^n=V_\sigma\oplus W$, and since $L(v_n)=1$ we have $L(W)=\O$. 
Therefore (ii) $\implies$ (iii).

Finally, assume (iii), and choose any $w\in W$ with $L(w)=1$.  We then have a splitting $W=\langle w\rangle\oplus \ker L|_W$. The submodule $\ker L|_W$ is free since $[\ker L|_W]=[W]-[\O]=0$, so we can choose a basis $\{u_{k+3},\ldots,u_n\}$ for $\ker L|_W$. Then 
\[\{v_1,\ldots,v_{k+1},w,u_{k+3},\ldots,u_n\}\] 
is a good basis containing $\sigma$. Therefore (iii) $\implies$ (i).

We now prove the second part of the proposition. By Lemma~\ref{lemma:partialbasissub}(a), $V_\sigma$ has a complement $U \subset \O^n$ such that $\O^n = V_\sigma \oplus U$.
In particular $\O=L(\O^n) = L(V_\sigma)+L(U)=I_\sigma+L(U)$.  If $I_\sigma=(0)$, this implies that $L(U) = \O$, so $U$ is a good complement
for $\sigma$.
Otherwise, choose a basis $\{u_{k+2},\ldots,u_n\}$ for $U$.  If $I_\sigma=\O$, there exists $v\in V_\sigma$ with $L(v)=1$. Then $W\coloneq\langle u_{k+2}+(1-L(u_{k+2}))v,u_{k+3},\ldots,u_n\rangle$ is a good complement for $\sigma$.

It remains to handle the case when $(0)\neq I_\sigma\neq \O$ and $k<n-2$, so $\Rank(U)\geq 2$.
Since $I_\sigma\neq \O$ we cannot have $L(u_i)=0$ for all $k+2\leq i\leq n$, so assume without loss of generality that $L(u_n) \neq 0$. Consider the $(n-2)$-simplex
\[\sigma^{n-2}\coloneq\{v_1,\ldots,v_{k+1},u_{k+3},\ldots,u_n\}\in \PB_n(\O).\]
Applying Lemma~\ref{lemma:complement2} with $I=\O$ gives a vertex $w$ in $\Link_{\PB_n(\O)}(\sigma)$ such that  $I_{\{w,u_n\}}=\O$ and $\{v_1,\ldots,v_{k+1},w,u_{k+3},\ldots,u_n\}$
is a basis for $\O^n$. It follows that
$W\coloneq \Span{w,u_{k+3},\ldots,u_n}$ is a good complement for $\sigma$.
\end{proof}

\subsubsection{Completing the proof of Step 2}
The pieces are now in place for the proof of Proposition~\ref{prop:BBprime}.

\begin{proof}[Proof of Proposition~\ref{prop:BBprime}]
Given a $k$-simplex $\sigma^k\in \PB_n(I)$ with $k\leq n-3$, 
Proposition~\ref{prop:charBprime}(2) implies that $\sigma$ is contained in $\PartialBasesGood_n(I)$. In other words, $\PBg_n(I)$ contains the $(n-3)$-skeleton of $\PartialBases_n(I)$. Thus $(\PartialBases_n(I),\PartialBasesGood_n(I))$ is $(n-3)$-connected, and every  
$(n-2)$-simplex $\sigma^{n-2}\in \PartialBases_n(I)$ has $\partial \sigma^{n-2}\subset \PartialBasesGood_n(I)$.
To prove that $(\PartialBases_n(I),\PartialBasesGood_n(I))$ is $(n-2)$-connected, it suffices to show that 
every $(n-2)$-simplex of $\PartialBases_n(I)$ that does not lie in $\PartialBasesGood_n(I)$ is homotopic
relative to its boundary into $\PartialBasesGood_n(I)$.

Let $\sigma^{n-2}$ be an $(n-2)$-simplex of $\PartialBases_n(I)$ that does not lie in $\PartialBasesGood_n(I)$. 
By Proposition~\ref{prop:charBprime}(2), we must have $(0)\subsetneq I_\sigma \subsetneq \O$.
Choose a vertex $v$ of $\sigma$ with $L(v)\neq 0$; if $I_\sigma\not\subset I$ then choose $v$ such that $L(v)\not\equiv 0\bmod{I}$.  Using Lemma~\ref{lemma:complement2}, choose a vertex $w$ of $\PartialBases_n(I)$
such that $\sigma^{n-2} \cup \{w\} \in \PartialBases_n(I)$ and $I_{\{v,w\}} = \O$.
Write $\sigma^{n-2}=\sigma_0 \cup \{v\}$, let $V=V_{\sigma_0}$ and let $W=\Span{v,w}$.  

The 1-dimensional complex $\PB_W(I)$, whose edges consist of $I$-bases for $W$, is contained in $\Link_{\PartialBases_n(I)}(\sigma_0)$. 
As we discussed following Proposition~\ref{prop:charBprime}, identifying $W$ with $\O^2$ induces an isomorphism 
$\PB_W(I) \iso \PB_2(I)$. We have assumed that $\PB_2(I)$ is connected, so  $\PB_W(I)$ is connected.

Fix a vertex $y$ of $\PB_W(I)$ such that $L(y) = 1$, and let $P$ be an embedded path in $\PB_W(I)$ from
$v$ to $y$.  Let $v = p_1,\ldots,p_m=y$ be an enumeration of the vertices of $P$.  For
$1 \leq i < m$, the set $\{p_i,p_{i+1}\}$ is a basis for $W$, so $(L(p_i),L(p_{i+1})) = L(W) = \O$.

Regarding $P$ as a triangulated closed interval that is mapped into $\Link_{\PartialBases_n(I)}(\sigma_0)$, 
we obtain an embedding of the join $\sigma_0 \ast P$ into $\PartialBases_n(I)$.  Its boundary is
\[\partial(\sigma_0\ast P)=\left(\partial \sigma_0\ast P\right) \cup \left(\sigma_0\ast \partial P\right).\]
Since $\partial P=\{v\}\sqcup \{y\}$, the second term $\sigma_0\ast\partial P$ consists of $\sigma_0 \ast \{v\}=\sigma$ 
and $\sigma_0 \ast \{y\}$.  Since $L(y)=1$, $\sigma_0\ast\{y\}$ is a partial good $I$-basis by Proposition~\ref{prop:charBprime}(2), so $\sigma_0\ast \{y\}\in\PartialBasesGood_n(I)$.

The first term $\partial\sigma_0 \ast P$ can be written as 
\[\partial \sigma_0\ast P = \bigsqcup_i \partial \sigma_0 \ast \{p_i,p_{i+1}\}.\]
Since $I_{\{p_i,p_{i+1}\}}=\O$, every $(n-2)$-simplex $\tau$ occurring in the join $\partial \sigma_0 \ast \{p_i,p_{i+1}\}$ 
has $I_\tau=\O$, and thus lies in $\PartialBasesGood_n(I)$ by Proposition~\ref{prop:charBprime}(2).  We conclude that $\sigma$ 
is the unique face of $\sigma_0\ast P$ that is not contained in $\PartialBasesGood_n(I)$.  The subdivided $(n-1)$-simplex 
$\sigma_0 \ast P$ thus provides the desired homotopy of $\sigma$ into $\PartialBasesGood_n(I)$.
\end{proof}

\subsection{Step 3: \texorpdfstring{$(\PBg_n(I),\PB_n(0))$}{(Bn-good(I),Bn(0))} is \texorpdfstring{$(n-2)$-connected}{(n-2)-connected}}
\label{section:step3}

Since every vertex $v\in \PB_n(0)$ has $L(v)=0$ or $L(v)=1$, every simplex $\sigma\in \PB_n(0)$ has $I_\sigma=(0)$ or $I_\sigma=\O$. By Proposition~\ref{prop:charBprime}(2), every such simplex is good, so $\PB_n(0)$ is contained in $\PBg_n(I)$.
The main result of this section is the following proposition. 

\begin{proposition}
\label{prop:BprimeB0}
Let $\O$ be a Dedekind domain, and fix $n\geq 3$. Assume that for all $1\leq m<n$ and all ideals $I\subset \O$, the complex $\PB_m(I)$ is CM of dimension $m-1$. Then for any ideal $I\subset \O$, the pair $(\PBg_n(I),\PB_n(0))$ is $(n-2)$-connected.
\end{proposition}
\noindent Before proving Proposition~\ref{prop:BprimeB0} we establish three results that will be necessary in the proof.

\subsubsection{Pushing into a subcomplex} 
The first ingredient is the following lemma on the connectivity of a pair.

\begin{lemma}
\label{lemma:pushing}
Let $A$ be a full subcomplex of a simplicial complex $X$, and fix $m \geq 0$. Assume
that for every $k$-simplex $\sigma^k \in X$ which is disjoint from $A$, the intersection
$\Link_X(\sigma^k) \cap A$ is $(m-k-1)$-connected.  Then the pair $(X,A)$ is $m$-connected.
\end{lemma}

The proof of Lemma~\ref{lemma:pushing} will use the language of posets.
Recall that a poset $P$ is said to be $d$-connected if its geometric realization $\abs{P}$ is $d$-connected. Similarly, the homology $\HH_{\ast}(P;\Z)$ is defined to be $\HH_{\ast}(\abs{P};\Z)$. Of course, a simplicial complex $X$ can be regarded as a poset $\Poset(X)$ whose elements are the simplices of $X$ under inclusion. The canonical homeomorphism $X\iso \abs{\Poset(X)}$ shows that $\Poset(X)$ is $d$-connected if and only if $X$ is, and there is a canonical isomorphism $\HH_{\ast}(X;\Z)\iso \HH_{\ast}(\Poset(X);\Z)$.

\begin{proof}[{Proof of Lemma~\ref{lemma:pushing}}]
We can replace $X$ by its $(m+1)$-skeleton without affecting the truth of the conclusion, 
so we can assume that $X$ is finite-dimensional.  Let $A$ be the full complex on the set of vertices $A^{(0)} \subset X^{(0)}$, so
$\Poset(A)=\Set{$\tau\in \Poset(X)$}{$\tau\subset A^{(0)}$}$.  Define $B\coloneq X^{(0)} \setminus A^{(0)}$.  
A simplex $\sigma\in \Poset(X)$ is 
disjoint from $A$ if $\sigma\cap A^{(0)}=\emptyset$, i.e.\ if $\sigma\subset B$; in 
this case, $\Link_X(\sigma)\cap A$ is $\Set{$\eta\in A$}{$\sigma\cup \eta\in \Poset(X)$}$.

We will use PL Morse theory as described in \cite{BestvinaPL} to study 
$(X,A) \cong (\abs{\Poset(A)},\abs{\Poset(X)})$.  In this context, a PL Morse
function is a function $\abs{\Poset(X)} \rightarrow \R$ with the following three
properties:
\begin{itemize}[nosep]
\item the restriction to each simplex is affine, and
\item the image of the vertex set is discrete, and
\item there are no horizontal edges, i.e.\ edges on which the function is constant.
\end{itemize}
Since the restriction to each simplex is affine, it is enough to define
the function on the $0$-skeleton $\Poset(X)$ of $\abs{\Poset(X)}$.  In fact, what we will
define is a map from $\Poset(X)$ to a totally ordered set whose image is finite;
the required map $\Poset(X) \rightarrow \R$ can then be constructed by embedding
the image in $\R$ in an order-preserving way.

Endow $\N \times -\N$ with the lexicographic order.  Our PL Morse function will
then be obtained via the above procedure from the map 
$f\colon \Poset(X)\to \N\times -\N$ defined by 
\[f(\sigma)=\big(\,\lvert\sigma\cap B\rvert,-\lvert\sigma\cap A^{(0)}\rvert\,\big).\]
Since $X$ is finite-dimensional, the image of $f$ is finite.  As for horizontal edges,
the edges of $\abs{\Poset(X)}$ are chains $\sigma \subsetneq \sigma'$ of simplices
$\sigma, \sigma' \in \Poset(X)$.  For these, we must have
$\abs{\sigma\cap B}<\abs{\sigma'\cap B}$ or 
$\abs{\sigma\cap A^{(0)}}<\abs{\sigma'\cap A^{(0)}}$, so $f(\sigma)\neq f(\sigma')$.  The
lack of horizontal edges follows. 

Note that $f^{-1}(\{0\}\times -\N)=\Set{$\sigma\in X$}{$\abs{\sigma \cap B}=0$}=\Poset(A)$. 
For any $\sigma\in \Poset(X)$, define the \emph{downward link} to be
\[\Link_\downarrow(\sigma)\coloneq \Set{$\tau\in \Poset(X)$}{$f(\tau)<f(\sigma)$ and either $\tau\subsetneq \sigma$ or $\tau\supsetneq \sigma$}.\]
From \cite[Proposition 2.7]{BestvinaPL} (which we regard as the
``fundamental theorem of PL Morse theory''), it follows that
if $\Link_\downarrow(\sigma)$ 
is $(m-1)$-connected for all $\sigma\in f^{-1}(\N_{>0}\times -\N)$, then $(\Poset(X),\Poset(A))$ 
is $m$-connected (here again we emphasize that the simplices of $X$ are the vertices
of $\abs{\Poset(X)}$).  We wish to verify this condition, so consider $\sigma \in f^{-1}(\N_{>0}\times -\N)$.

We can uniquely write $\sigma$ as the disjoint union of $\sigma_B\coloneq \sigma\cap B$ 
and $\sigma_A\coloneq \sigma\cap A^{(0)}$, where $\abs{\sigma_B}>0$ by assumption.    Define subposets of $\Poset(X)$ by
\begin{align*}
T&\coloneq \Set{$\tau\in \Poset(X)$}{$\tau\subsetneq \sigma$, $\abs{\tau\cap B}<\abs{\sigma\cap B}$},\\
R&\coloneq \Set{$\rho\in \Poset(X)$}{$\rho\supsetneq \sigma$, $\abs{\rho\cap B}=\abs{\sigma\cap B}$, $\lvert\rho\cap A^{(0)}\rvert>\lvert\sigma\cap A^{(0)}\rvert$}.
\end{align*} 
The downward link $\Link_\downarrow(\sigma)$ is the sub-poset $T\cup R$ of $\Poset(X)$; since 
every $\tau\in T$ and every $\rho\in R$ satisfy $\tau\subsetneq \rho$, in fact the downward link is 
the join $T\ast R$. It is well-known that the join of an $a$-connected poset with a $b$-connected poset is $(a+b+2)$-connected (see e.g.\ \cite[Example 8.1]{QuillenPoset}).

Now, either $\sigma_A \neq \emptyset$ or $\sigma_A=\emptyset$.  In the first case, the poset $T$
is contractible via the homotopy that takes $\tau$ to $\tau \cup \sigma_A$ and then to $\sigma_A$.
We conclude that in this case $\Link_\downarrow(\sigma)=T\ast R$ is contractible, and in particular
is $(m-1)$-connected.

So suppose that $\sigma_A=\emptyset$.  In other words, $\sigma$ is disjoint from $A$. 
In this case, $T$ is the poset of proper nonempty subsets of the $(k+1)$-element set $\sigma^k$, so $T\iso \Poset(\partial \Delta^k)\simeq S^{k-1}$; in particular, $T$ is $(k-2)$-connected.  The poset $R$ consists of $\rho\in \Poset(X)$ that can be written as 
$\rho=\sigma\cup \eta$ for some nonempty $\eta\in \Poset(A)$.  The assignment 
$\sigma\cup \eta\mapsto \eta$ gives an isomorphism of $R$ with $\Poset(\Link_X(\sigma)\cap A)$. Our assumption guarantees that $R$ is $(m-k-1)$-connected, so $\Link_\downarrow(\sigma)=T\ast R$ 
is $(m-1)$-connected, as desired.
\end{proof}

\subsubsection{Topology of nicely fibered complexes}
The second ingredient we will need concerns the topology of simplicial complexes that are ``fibered'' over a CM complex in the following sense.

\begin{definition}
\label{def:fibersCMsplit}
Let $F\colon X\to Y$ be a simplicial map between simplicial complexes.  Let $C\subset X^{(0)}$ be a subset of the vertices of $X$ with $F(C)=Y^{(0)}$.

We say that $X$ is \emph{fibered over $Y$ by $F$ with core $C$} if the following condition holds:
a subset $U=\{x_0,\ldots,x_k\}\subset X^{(0)}$ forms a $k$-simplex of $X$ if and only if 
\begin{enumerate}[topsep=2pt,itemsep=1pt,parsep=1pt]
\item[(1)] $F(U)=\{F(x_0),\ldots,F(x_k)\}$ is a $k$-simplex of $Y$, and
\item[(2)] if $F(U)$ is a maximal simplex of $Y$, then $U\cap C\neq\emptyset$.
\end{enumerate}
If $C=X^{(0)}$, we say $X$ is \emph{fully fibered over $Y$ by $F$}; in this case condition (2) is vacuous.
\end{definition}

The main result concerning fibered complexes is as follows.

\begin{lemma}
\label{lemma:fibersCMsplit}
Assume that $X$ is fibered over $Y$ by $F\colon X\to Y$ with core $C\subset X^{(0)}$. If $Y$ is CM of dimension $d$ then $X$ is CM of dimension $d$.
\end{lemma}

To prove Lemma~\ref{lemma:fibersCMsplit} we will make use of a theorem of Quillen that describes the topology of a poset map in terms of the topology of the fibers.

Given a poset $A$, recall that the \emph{height} $\height(a)$ of an element $a$
is the largest $k$ for which there exists a chain $a_0\lneq a_1\lneq \cdots \lneq a_k=a$. 
When a $k$-simplex $\sigma^k\in X$ is regarded as an element of the poset $\Poset(X)$, it has height $k$.
Given a map of posets $F\colon A\to B$ and element $b\in B$, the fiber $F_{\leq b}$ 
is the subposet of $A$ defined by $F_{\leq b}\coloneq \{a\in A \,|\, F(a)\leq b\}$.
 
A poset map $F\colon A\to B$ is {\em strictly increasing} if $a<a'$ implies $F(a)<F(a')$.
Finally, a poset $P$ is said to be CM of dimension $d$ if its geometric realization $\abs{P}$ is CM of dimension $d$ as a simplicial complex.  A simplicial complex $X$ is CM of dimension $d$ if and only if $\Poset(X)$ is CM of dimension $d$ (this is not automatic since being CM is not homeomorphism-invariant; nevertheless it follows from \cite[(8.5)]{QuillenPoset}). 

\begin{theorem}[{Quillen~\cite[Theorem 9.1, Corollary 9.7]{QuillenPoset}}]
\label{theorem:quillen}
Let $F\colon A\to B$ be a strictly increasing map of posets.  Assume that $B$ 
is CM of dimension $d$ and that for all $b\in B$, the fiber $F_{\leq b}$ 
is CM of dimension $\height(b)$.  Then $A$ is CM of dimension $d$ and $F_{\ast}\colon \widetilde{\HH}_d(A;\Z)\to \widetilde{\HH}_d(B;\Z)$ is surjective.
\end{theorem}

We also record the following observation, which follows directly from Definition~\ref{def:fibersCMsplit}.
\begin{lemma}
\label{lemma:linksfibered}
Let  $X$ be fibered over $Y$ by $F$ with core $C$, and consider a simplex $\sigma\in X$. 
\begin{itemize}[nosep]
\item If $\sigma\cap C\neq \emptyset$, then the link $\Link_X(\sigma)$ is fully fibered over $\Link_Y(F(\sigma))$ by $F$.
\item If $\sigma\cap C=\emptyset$, then the link $\Link_X(\sigma)$ is fibered over $\Link_Y(F(\sigma))$ by $F$ 
with core $C\cap \Link_X(\sigma)^{(0)}$.
\end{itemize}
\end{lemma}

\begin{proof}[Proof of Lemma~\ref{lemma:fibersCMsplit}]
We prove the lemma by induction on $d$.  The base case $d=0$ is trivial, so assume that $d>0$ and that the
lemma is true for all smaller values of $d$.  We must prove two things about $X$: that for every $\sigma^k\in X$ the link $\Link_X(\sigma^k)$ is CM of dimension $d-k-1$, and that $X$ itself is $d$-spherical. By condition (1), the $k$-simplex $\sigma^k\in X$ projects to a $k$-simplex $F(\sigma^k)\in Y$. By Lemma~\ref{lemma:linksfibered}, the complex $\Link_X(\sigma^k)$ is fibered over $\Link_Y(F(\sigma^k))$ by $F$ (with some core). Since $Y$ is CM of dimension $d$, the complex $\Link_Y(F(\sigma^k))$ is CM of dimension $d-k-1$. Applying the inductive hypothesis, we conclude that $\Link_X(\sigma^k)$ is CM of dimension $d-k-1$, as desired.

It remains to prove that $X$ is $d$-spherical. We begin by proving the lemma in the special case when $X$ is fibered over a single simplex $Y=\Delta^d$ with vertices $\{y_0,\ldots,y_d\}$.  Define $Z$ to be the complex with $Z^{(0)}=X^{(0)}$
and where $U \subset Z^{(0)}$ forms a simplex if and only if $U$ contains at most one element from each
$F^{-1}(y_i)$.  The complex $Z$ is the join of the $0$-dimensional complexes $F^{-1}(y_i)$.
Since $F(C)=Y^{(0)}$, each $F^{-1}(y_i)$ is nonempty and thus $0$-spherical.  The join of an $r$-spherical 
and an $s$-spherical complex is $(r+s+1)$-spherical \cite[Example 8.1]{QuillenPoset},
so the $(d+1)$-fold join $Z=F^{-1}(y_0)\ast\cdots\ast F^{-1}(y_d)$ is $d$-spherical.

Observe that $X \subset Z$.  Since $Z$ is $d$-spherical, it is enough to prove that the
pair $(Z,X)$ is $(d-1)$-connected.
Since $Y$ is itself a $d$-simplex, the only maximal simplex of $Y$ is $Y$ itself, and hence
condition (2) applies only to $d$-simplices of $X$.  Therefore $X$ and $Z$ coincide in dimensions up to $d-1$ 
and $Z$ is obtained from $X$ by adding certain $d$-simplices $\sigma^d\in Z$ with $\sigma^d\not\in X$. 
It thus suffices to show for each such $\sigma^d$ that the boundary $\partial \sigma^d\subset X$ is null-homotopic in $X$.

By condition (1), such a $d$-simplex can be written as $\sigma^d=\{x_0,\ldots,x_d\}$ with $F(x_i)=y_i$ for $0 \leq i \leq d$, 
and by condition (2) we must have $x_i\not\in C$ for all $0 \leq i \leq d$.  Since $F(C)= Y^{(0)}$, for each $0 \leq i \leq d$
we can choose $x'_i\in C$ with $F(x'_i)=y_i$.  Consider the full subcomplex of $Z$ on $\{x_0,x'_0,\ldots,x_d,x'_d\}$.  By
construction, this is the $(d+1)$-fold join of the $0$-spheres $\{x_i,x_i'\}$, and thus is homeomorphic to a $d$-sphere.  Every simplex of this $d$-sphere \emph{except} 
$\sigma$ contains some $x'_i$, and thus lies in $X$.  Therefore the full subcomplex of $X$ on $\{x_0,x'_0,\ldots,x_d,x'_d\}$ 
is a subdivided $d$-disk $D$ with $\partial D=\partial \sigma^d$, as desired.

This proves that $X$ is $d$-spherical if it is fibered over a $d$-simplex with any core; combined with the first paragraph, $X$ is CM of dimension $d$ in this case.

We now prove the lemma for a general CM complex $Y$ of dimension $d$.  By condition (1), the map $F$ induces a height-preserving poset map $F\colon \Poset(X)\to \Poset(Y)$; we will verify the conditions of Theorem~\ref{theorem:quillen} for $F$.

Fix a simplex $\tau^k\in \Poset(Y)$. Let $T\simeq \Delta^k$ be the subcomplex of $Y$ determined by $\tau^k$ and let $\Ttop$ be the subcomplex of $X$ consisting of $\sigma^\ell\in X$ with $F(\sigma^\ell)\subset \tau^k$.  
The fiber $F_{\leq \tau^k}$ is the subposet of $\Poset(X)$ consisting of simplices of $\Ttop$, so $F_{\leq\tau^k}=\Poset(\Ttop)$. Since $F$ is height-preserving, if $k<d$ then no $d$-simplex is contained in $\Ttop$, so condition (2) is vacuous; in this case $\Ttop$ is fully fibered over $T\simeq \Delta^k$ by $F$. If $\tau^d$ is a $d$-simplex,  condition (2) is not vacuous, and in this case $\Ttop$ is fibered over $T\simeq\Delta^d$ by $F$ with core $C\cap \Ttop^{(0)}$. In either case the base $T$ is a single simplex, so the special case of the lemma implies that $\Ttop$ is CM of dimension $\height(\tau^k)=k$, and thus so is $F_{\leq\tau^k}=\Poset(\Ttop)$.
Since $\Poset(Y)$ is CM of dimension $d$ by assumption, this verifies the hypotheses of Theorem~\ref{theorem:quillen} for the map $F\colon \Poset(X)\to \Poset(Y)$. Theorem~\ref{theorem:quillen} implies that $\Poset(X)$ is CM of dimension $d$, so $X$ is CM of dimension $d$ as well.
\end{proof}

\subsubsection{A criterion for being fibered}
The final ingredient we need is the following condition to verify that a complex is fibered over another.

\begin{lemma}
\label{lemma:fibersonlyd}
Let $X$ and $Y$ be $d$-dimensional simplicial complexes with the property that every simplex is contained in a $d$-simplex. Let $F\colon X\to Y$ be a simplicial map and $C\subset X^{(0)}$ be a subset such that the following two conditions hold.
\begin{enumerate}[label={\normalfont (\roman*)}, topsep=2pt,itemsep=1pt,parsep=1pt]
\item A subset $U=\{x_0,\ldots,x_d\}$ forms a $d$-simplex of $X$\newline  $\iff$ $F(U)$ is a $d$-simplex of $Y$ and $U\cap C\neq \emptyset$.
\item For each $d$-simplex $\tau^d\in Y$ there exists $U\subset C$ with $F(U)=\tau^d$.
\end{enumerate}
Then $X$ is fibered over $Y$ by $F$ with core $C$.
\end{lemma}
\begin{proof}
The maximal simplices of $Y$ are precisely the $d$-simplices, so when $\abs{U}=d+1$ condition (i) is equivalent to conditions (1) and (2) of Definition~\ref{def:fibersCMsplit}. Since every vertex of $Y$ is contained in a $d$-simplex, condition (ii) implies that $F(C)=Y^{(0)}$. It remains to verify conditions (1) and (2) for subsets $U\subset X^{(0)}$ with $\abs{U}<d+1$.

Every simplex $\tau^k\in X$ is contained in a $d$-simplex $\sigma^d\in X$. By condition (i), its image $F(\sigma^d)$ is a $d$-simplex of $Y$, so $F(\tau^k)$ is a $k$-simplex of $Y$. Conversely, consider a subset $U=\{x_0,\ldots,x_k\}$ with $k<d$ such that $F(U)$ is a $k$-simplex of $Y$. Then $F(U)$ is contained in a $d$-simplex $\eta^d=F(U)\cup\{y_{k+1},\ldots,y_d\}\in Y$. Using the assumption that $F(C)=Y^{(0)}$, choose lifts $x_i\in C$ with $F(x_i)=y_i$ for $k+1\leq i\leq d$. By construction, $U'\coloneq\{x_0,\ldots,x_d\}$ has $F(U')=\eta^d$ and $U'\cap C\neq \emptyset$, so by condition (i) $U'$ is a $d$-simplex of $X$. It follows that $U$ is a $k$-simplex of $X$, as desired.
\end{proof}

\subsubsection{Completing the proof of Step 3}
We are now ready to prove Proposition~\ref{prop:BprimeB0}.

\begin{proof}[Proof of Proposition~\ref{prop:BprimeB0}]
We want to show that $(\PBg_n(I),\PB_n(0))$ is $(n-2)$-connected. We will prove this by applying Lemma~\ref{lemma:pushing} with $X = \PartialBasesGood_n(I)$ and $A = \PartialBases_n(0)$ and $m=n-2$. 

 To apply Lemma~\ref{lemma:pushing} we need to verify the two hypotheses: 
first, that $\PartialBases_n(0)$ is a full subcomplex of $\PartialBasesGood_n(I)$, and second, that for any $k$-simplex
$\sigma^k$ of $\PartialBasesGood_n(I)$ that is disjoint from $\PartialBases_n(0)$, the intersection
\[\Link_0(\sigma^k)\coloneq \PB_n(0)\cap \Link_{\PartialBasesGood_n(I)}(\sigma^k)\] is $(n-k-3)$-connected. 

We begin by proving that $\PB_n(0)$ is the full subcomplex of $\PB_n(\O)$ on the vertex set \[\PB_n(0)^{(0)}\coloneq \left\{v\in \PB_n(\O)\,|\,L(v)=0\text{ or }L(v) = 1\right\}.\]
Let $\sigma=\{v_1,\ldots,v_k\}\in \PB_n(\O)$ be a partial basis with $v_i\in \PB_n(0)^{(0)}$ for $1\leq i\leq k$. We have either $I_\sigma=\O$ (if some $v_i$ has $L(v_i)=1$) or $I_\sigma=(0)$ (if all $v_i$ have $L(v_i)=0$). By Proposition~\ref{prop:charBprime}(2), $\sigma$ is contained in a good basis. By Proposition~\ref{prop:charBprime}(1)(ii) this implies that $\sigma$ is contained in a good $(0)$-basis, so $\sigma\in \PB_n(0)$. This proves that $\PB_n(0)$ is a full subcomplex of $\PB_n(\O)$, and thus of $\PBg_n(I)$, verifying the first hypothesis. 

For the second hypothesis, consider a $k$-simplex $\sigma^k\in \PBg_n(I)$ that is disjoint from $\PB_n(0)^{(0)}$; we will prove that $\Link_0(\sigma^k)$ is CM of dimension $n-k-2$. Using Proposition~\ref{prop:charBprime}(1)(iii), choose a good complement $W$.  By definition, we have $\O^n=V_\sigma\oplus W$ and $L(W)=\O$. Let $\pi\colon \O^n\onto W$ be the projection with kernel $V_\sigma$. 

We will prove the following claims, using the terminology of Definition~\ref{def:fibersCMsplit}.
\begin{enumerate}[nosep]
\item If $I_\sigma\neq \O$, then $\Link_0(\sigma)$ is fully fibered over $\PB_W(I_\sigma)$ by $\pi$.
\item 
If $I_\sigma=\O$, then $\Link_0(\sigma)$ is fibered over $\PB_W(\O)$ by $\pi$ with core \[C\coloneq \{ v\in \Link_0(\sigma)\,|\, L(v)=1\}\subset \Link_0(\sigma)^{(0)}.\]
\end{enumerate}
The bulk of the argument does not depend on whether $I_\sigma=\O$ or not, so we begin without making any assumption on $I_\sigma$.

By definition, every simplex of $\PB_W(I_\sigma)$ is contained in an $(n-k-2)$-simplex, corresponding to an $I_\sigma$-basis of $W$. Similarly, if $\eta\in \Link_0(\sigma^k)$ then $\sigma\cup\eta\in \PBg_n(I)$ is a partial good $I$-basis.  By Proposition~\ref{prop:charBprime}(1)(ii)  we can enlarge this to a good $I$-basis $\sigma\cup\tau$ with $\tau\in \PB_n(0)$. Therefore every simplex $\eta\in \Link_0(\sigma^k)$ is contained in an $(n-k-2)$-simplex $\tau^{n-k-2}\in \Link_0(\sigma^k)$. This means that in verifying the claim, we can work only with $(n-k-2)$-simplices, and appeal to Lemma~\ref{lemma:fibersonlyd} for the general case.

A subset $U=\{u_{k+2},\ldots,u_n\}\subset \PB_n(0)^{(0)}$ forms an $(n-k-2)$-simplex of $\Link_0(\sigma^k)$ if $\sigma\cup U=\{v_1,\ldots,v_{k+1}, u_{k+2},\ldots,u_n\}$ is a good $I$-basis, or equivalently if $\sigma\cup U$ is a good basis. Since $\sigma$ is a basis of $V_\sigma$ and $W$ is a complement to $V_\sigma$, we know that $\sigma\cup U$ is a basis for $\O^n$ if and only if $\pi(U)$ is a basis for $W$. Moreover, any $u\in \PB_n(0)$ has $L(u)=0$ or $L(u)=1$, which implies $L(\pi(u))\equiv 0\bmod{I_\sigma}$ or $L(\pi(u))\equiv 1\bmod{I_\sigma}$. Therefore $\sigma\cup U$ is a basis for $\O^n$ if and only if $\pi(U)$ is an $I_\sigma$-basis for $W$.

We now prove the two cases of the claim separately. First, assume that $I_\sigma=\O$. Since $\sigma$ is disjoint from $\PB_n(0)$ by assumption, a basis $\sigma\cup U$ with $U\subset \PB_n(0)^{(0)}$ is good if and only if $U\cap C\neq \emptyset$. This means that $U$ forms an $(n-k-2)$-simplex of $\Link_0(\sigma)$ if and only if $\pi(U)$ is an $(n-k-2)$-simplex of $\PB_W(\O)$ and $U\cap C\neq \emptyset$, verifying condition (i) of Lemma~\ref{lemma:fibersonlyd}.

It remains to verify condition (ii) of Lemma~\ref{lemma:fibersonlyd}. Given a basis $\tau=\{w_{k+2},\ldots,w_n\}$ of $W$, choose for each $i=k+2,\ldots,n$ some $v'_i\in V_\sigma$ with $L(v'_i)=1-L(w_i)$, and define $u_i=w_i+v'_i$. By construction we have $L(u_i)=1$ and $\pi(u_i)=\pi(w_i)+\pi(v'_i)=w_i$. Setting $U=\{u_{k+2},\ldots,u_n\}$ we see that $\pi(U)$ is the basis $\tau$ of $W$, and therefore $U\subset C$. We can therefore apply Lemma~\ref{lemma:fibersonlyd} to conclude that $\Link_0(\sigma)$ is  fibered over $\PB_W(I_\sigma)$ by $\pi$ with core $C$, verifying Claim 1.

Second, assume that $I_\sigma\neq \O$. In this case \emph{any} basis $\sigma\cup U$ with $U\subset \PB_n(0)^{(0)}$ must be good, because if $L(u_i)=0$ for all $k+2\leq i\leq n$ we would have $L(\O^n)=I_\sigma\neq \O$. Therefore $\sigma\cup U$ is a good basis if and only if $\pi(U)$ is an $I_\sigma$-basis of $W$. This verifies condition (i) of Lemma~\ref{lemma:fibersonlyd}, so it remains to check condition (ii).

Let  $\tau=\{w_{k+2},\ldots,w_n\}$ be an $I_\sigma$-basis of $W$. For $i=k+2,\ldots,n$, if $L(w_i)\equiv 0\bmod{I_\sigma}$, choose $v'_i\in V_\sigma$ with $L(v'_i)=-L(w_i)$; if $L(w_i)\equiv 1\bmod{I_\sigma}$, choose $v'_i\in V_\sigma$ with $L(v'_i)=1-L(w_i)$. In either case, define $u_i=w_i+v'_i$; by construction, we have $L(u_i)=0$ or $L(u_i)=1$ and  $\pi(u_i)=\pi(w_i)+\pi(v'_i)=w_i$. Setting $U=\{u_{k+2},\ldots,u_n\}$ we see that $\pi(U)$ is the $I_\sigma$-basis $\tau$ of $W$; therefore $U\subset \Link_0(\sigma)$. We again apply Lemma~\ref{lemma:fibersonlyd} to conclude that $\Link_0(\sigma)$ is fully fibered over $\PB_W(I_\sigma)$ by $\pi$, verifying Claim 2.

Now that Claims 1 and 2 are established, we finish the proof of the proposition. Since $W\iso \O^{n-k-1}$,  Remark~\ref{remark:nonuniqueBnI} yields an isomorphism $\PB_W(I_\sigma)\iso \PB_{n-k-1}(I_\sigma)$. Our hypothesis guarantees that $\PB_{n-k-1}(I_\sigma)$ is CM of dimension $n-k-2$, so the same is true of $\PB_W(I_\sigma)$. We established in Claims 1 and 2 above that $\Link_0(\sigma)$ is fibered over $\PB_W(I_\sigma)$ by $\pi$ (either with core $C$ or with core $\Link_0(\sigma)^{(0)}$). Lemma~\ref{lemma:fibersCMsplit} thus implies that $\Link_0(\sigma)$ is CM of dimension $n-k-2$; in particular $\Link_0(\sigma)$ is $(n-k-3)$-connected. This verifies the hypothesis of Lemma~\ref{lemma:pushing} for the pair $(\PBg_n(I),\PB_n(0))$ with $m=n-2$; we conclude that $(\PBg_n(I),\PB_n(0))$ is $(n-2)$-connected, as desired.
\end{proof}
\begin{remark}
\label{remark:trouble}
Many of the ingredients of the proof that $\Link_0(\sigma^k)$ is $(n-k-3)$-connected were introduced only for this step; this step thereby accounts for much of the complexity of our proof of Theorem~\ref{maintheorem:partialbasescm}.
\begin{itemize}[nosep]
\item Even if we had only been interested in $\PB_n(\O)$, we were forced in the proof of Proposition~\ref{prop:BprimeB0} to consider the complex $\PB_{n-k-1}(I_\sigma)$. This is the entire reason that we introduced the complexes $\PB_n(I)$ relative to nontrivial ideals of $\O$.
\item In the proof of Proposition~\ref{prop:BprimeB0} we needed the assumption that $\PB_m(I_\sigma)$ was CM of dimension $m-1$ for $m<n$. This is the reason that we are forced to prove in Theorem~\ref{thm:BnICM} that $\PB_n(I)$ is Cohen--Macaulay of dimension $n-1$, rather than simply that $\PB_n(I)$ is $(n-1)$-spherical. This weaker statement would not suffice as an inductive hypothesis; Lemma~\ref{lemma:fibersCMsplit} is not true if one only assumes that $Y$ is $d$-spherical.
\item In the case $I_\sigma=\O$ above, we applied Lemma~\ref{lemma:fibersCMsplit} in a case where the core $C$ only contained some of the vertices of the domain. This is the reason that we introduced the notion of ``fibered with core $C$''; we will apply Lemma~\ref{lemma:fibersCMsplit} in Step 4 and in the proof of Theorem~\ref{maintheorem:partialbasescm} 
as well, but only in the fully fibered case.
\end{itemize}
\end{remark}

\subsection{Step 4: \texorpdfstring{$\PB_n(0)$}{Bn(0)} is \texorpdfstring{$(n-2)$-connected}{(n-2)-connected}}
\label{section:step4}

The main result of this section is the following proposition. 

\begin{proposition}
\label{prop:B0connected}
Let $\O$ be a Dedekind domain, and assume that $\PB_{n-1}(\O)$ is CM of dimension $n-2$. Then $\PB_n(0)$ is $(n-2)$-connected.
\end{proposition}
\begin{proof}
Define $Z_n$ to be the full subcomplex of $\PartialBases_n(0)$ spanned by those vertices $v$ of $\PartialBases_n(0)$ with $L(v)=0$. The complex $Z_n$ consists of partial bases contained in the summand $\ker L\iso \O^{n-1}$.  In fact, by Lemma~\ref{lemma:partialbasissub}(b) these are precisely the partial bases of the summand $\ker L$, so we have an isomorphism $Z_n\iso \PB_{n-1}(\O)$. Therefore our assumption guarantees that $Z_n$ is CM of dimension $n-2$.

Consider a $k$-simplex $\sigma^k=\{v_0,v_1,\ldots,v_k\}\in \PartialBases_n(0)$ disjoint from $Z_n$, i.e.\ with $L(v_0)=\cdots=L(v_k)=1$. Let $\eta^{k-1}=\{v_1-v_0,\ldots,v_k-v_0\}$ and $\tau^k = \{v_0\}\cup \eta\in \PartialBases_n(0)$. Since $V_\sigma=V_\tau$, the links $\Link_{\PartialBases_n(0)}(\sigma^k)$ and $\Link_{\PartialBases_n(0)}(\tau^k)$ coincide.  Furthermore, given $\nu^{n-k-2}\in Z_n$ the set $\{v_0\}\cup \eta\cup\nu$ is a basis for $\O^n$ if and only if $\eta\cup\nu$ is a basis for $\ker L$. In other words,
\[\Link_{\PartialBases_n(0)}(\tau^k) \cap Z_n=\Link_{Z_n}(\eta^{k-1}).\]
 
Since $Z_n$ is CM of dimension $n-2$, we know that $\Link_{Z_n}(\eta^{k-1})$ is CM of dimension $(n-2)-(k-1)-1=n-k-2$, and is thus $(n-k-3)$-connected.
We now apply Lemma~\ref{lemma:pushing} with $X=\PartialBases_n(0)$ and $A=Z_n$ and $m=n-2$; the conclusion is
that $(\PB_n(0),Z_n)$ is $(n-2)$-connected.

Fix a vector $v\in \O^n$ with $L(v)=1$.  Since $\{u_2,\ldots,u_n\}\subset \ker L$ is a basis for $\ker L$ if and only if $\{v,u_2,\ldots,u_n\}$ is a basis for $\O^n$, every simplex of $Z_n$ is contained in $\Link_{\PB_n(0)}(v)$.
We conclude that the image of $Z_n\into \PB_n(0)$ 
can be contracted to $v$.  Together with the fact that $(\PB_n(0),Z_n)$ is $(n-2)$-connected, this
implies that $\PB_n(0)$ is $(n-2)$-connected, as desired.
\end{proof}

\subsection{Assembling the proof: \texorpdfstring{$\PB_n(I)$}{Bn(I)} is CM of dimension \texorpdfstring{$n-1$}{n-1}}
\label{section:BnICM}

We are finally ready to combine the steps of the previous sections and prove Theorem~\ref{thm:BnICM}. 
\begin{proof}[Proof of Theorem~\ref{thm:BnICM}]
We prove by induction on $n$ that $\PB_n(I)$ is CM of dimension $n-1$ for all ideals $I\subset \O$. The base cases are $n=1$ and $n=2$. Proposition~\ref{prop:basecasen1}, which holds for any Dedekind domain $\O$, states that $\PB_1(I)$ is CM of dimension 0.  Proposition~\ref{prop:basecasen2} states that under precisely our hypotheses, $\PB_2(I)$ is CM of dimension~1.

Now fix $n\geq 3$ and an ideal $I\subset \O$, and assume that $\PB_m(J)$ is CM of dimension $m-1$ for all $m<n$ and all $J\subset \O$.
Under these assumptions, Proposition~\ref{prop:BBprime} states that $(\PB_n(I),\PBg_n(I))$ is $(n-2)$-connected; Proposition~\ref{prop:BprimeB0} states that $(\PBg_n(I),\PB_n(0))$ is $(n-2)$-connected; and Proposition~\ref{prop:B0connected} states that $\PB_n(0)$ itself is $(n-2)$-connected. Together these imply that the $(n-1)$-dimensional complex $\PB_n(I)$ is $(n-2)$-connected, and thus $(n-1)$-spherical.

It remains to prove that for every $k$-simplex $\sigma^k\in\PartialBases_n(I)$, the space 
$\Link_{\PB_n(I)}(\sigma^k)$ is CM of dimension $n-k-2$. If $k=n-1$ this is vacuous.
If $k=n-2$, we must prove that $\Link_{\PB_n(I)}(\sigma^k)$ is nonempty. But the partial $I$-basis $\sigma^{n-2}$ is by definition contained in an $I$-basis $\sigma^{n-2}\cup\{v\}$, so $\Link_{\PB_n(I)}(\sigma^k)$ contains the $0$-simplex $\{v\}$.

Now assume that $k<n-2$. By Lemma~\ref{prop:charBprime}(1), $\sigma$ has a good complement $W$. Let $\pi\colon \O^n\onto W$ be the projection with kernel $V_\sigma$. We will prove that $\Link_{\PB_n(I)}(\sigma)$ is fully fibered over $\PB_W(I+I_\sigma)$ by $\pi$ using Lemma~\ref{lemma:fibersonlyd}.

By definition, every simplex of $\Link_{\PB_n(I)}(\sigma)$ or $\PB_W(I+I_\sigma)$ is contained in an $(n-k-2)$-simplex. Given a subset $U=\{u_{k+2},\ldots,u_n\}$ with $L(u_i)\equiv 0\text{ or }1\bmod{I}$ for each $i=k+2,\ldots,n$, the set $\sigma\cup U$ is a basis of $\O^n$ if and only if $\pi(U)$ is a basis of $W$. Moreover $L(\pi(u))\equiv L(u)\bmod{I_\sigma}$, so in this case $\pi(U)$ is an $(I+I_\sigma)$-basis of $W$. Therefore such a subset $U$ is an $(n-k-2)$-simplex of $\Link_{\PB_n(I)}(\sigma)$ if and only if $\pi(U)$ is an $(n-k-2)$-simplex of $\PB_W(I+I_\sigma)$, verifying condition (i) of Lemma~\ref{lemma:fibersonlyd}.

It remains to check condition (ii) of Lemma~\ref{lemma:fibersonlyd}, so consider an $(I+I_\sigma)$-basis $\tau=\{w_{k+2},\ldots,w_n\}$ of $W$. For each $i$ we have $L(w_i)\equiv 0\text{ or }1\bmod{I+I_\sigma}$, so there exists $v_i\in V_\sigma$ such that $L(w_i+v_i)\equiv 0\text{ or }1\bmod{I}$. The set $U\coloneq \{w_{k+2}+v_{k+2},\ldots,w_n+v_n\}$ projects under $\pi$ to $\tau$, so it is an $(n-k-2)$-simplex of $\Link_{\PB_n(I)}(\sigma)$. 

 Lemma~\ref{lemma:fibersonlyd} now implies that  $\Link_{\PB_n(I)}(\sigma)$ is fully fibered over $\PB_W(I+I_\sigma)$ by $\pi$. Since $W\iso \O^{n-k-1}$, the complex $\PB_W(I+I_\sigma)$ is isomorphic to $\PB_{n-k-1}(I+I_\sigma)$, which we have assumed by induction  is CM of dimension $n-k-2$. Applying Lemma~\ref{lemma:fibersCMsplit} shows that $\Link_{\PB_n(I)}(\sigma^k)$ is CM of dimension $n-k-2$, as desired.

We have proved that $\PB_n(I)$ is $(n-1)$-spherical and that $\Link_{\PB_n(I)}(\sigma^k)$ is CM of dimension $n-k-2$ for every $\sigma^k\in \PB_n(I)$, so $\PB_n(I)$ is CM of dimension $n-1$. Since the ideal $I$ was arbitrary, this concludes the inductive step, and thus concludes the proof of Theorem~\ref{thm:BnICM}.
\end{proof}

\section{Integral apartments}
\label{section:integralapartments}

In this section we use Theorem~\ref{maintheorem:partialbasescm} to prove Theorem~\ref{maintheorem:integralapartments}.

\begin{proof}[{Proof of Theorem~\ref{maintheorem:integralapartments}}]
Let $C_\ast$ be the reduced chain complex of the simplicial complex $\PB_n(\O)$, so $C_{-1}\iso \Z$. 
We define the \emph{integral apartment class map} 
\[\phi\colon C_{n-1}\to \St_n(K)\] 
as follows. An ordered  basis $\vv=(v_1,\ldots,v_n)$ of $\O^n$ determines an $(n-1)$-simplex of $\PB_n(\O)$, and thus a generator $[\vv]=[v_1,\ldots,v_n]$ of $C_{n-1}$.  Reordering the basis simply changes the orientation of this simplex, so we have $[v_{\tau(1)},\ldots,v_{\tau(n)}]=(-1)^\tau[v_1,\ldots,v_n]$ for $\tau\in S_n$; there are no other relations between the elements $[\vv]\in C_{n-1}$.
Given $\vv$, define lines $L_1,\ldots,L_n$ in $K^n$ via the formula $L_i \coloneq \spn_K(v_i)$.  Since $\{v_1,\ldots,v_n\}$ is a basis of $\O^n$, Lemma~\ref{lemma:summandsubspace}(c) implies  that $L_i\cap \O^n=\spn_\O(v_i)$ and 
 $\O^n=\spn_\O(v_1)\oplus\cdots\oplus \spn_\O(v_n)$, so the frame $\LL_{\vv}\coloneq \{L_1,\ldots,L_n\}$ is integral.   

Let $A_{\vv}\coloneq A_{\LL_{\vv}}$ be the integral apartment determined by this integral frame, and define $\phi([\vv])=[A_{\vv}]\in \St_n(K)$. Reordering the basis vectors gives the same frame, and thus the same apartment, but possibly with reversed orientation: $[A_{\tau(\vv)}]=(-1)^\tau [A_{\vv}]$. This shows that the map $\phi$ is well-defined. Our goal is to prove that $\phi$ is surjective, i.e.\ that $\St_n(K)$ is generated by integral apartment classes. We will do this by factoring the map $\phi$ as follows.

For a simplicial complex $X$, the isomorphism $b\colon \HH_{\ast}(X;\Z) \stackrel{\iso}{\rightarrow} \HH_{\ast}(\Poset(X);\Z)$ determined by the homeomorphism between $X$ and its barycentric subdivision is induced by the chain map \[[\{x_1,\ldots,x_k\}]\mapsto\sum_{\tau\in S_k}(-1)^\tau \big[\{x_{\tau(1)}\}\subset \{x_{\tau(1)},x_{\tau(2)}\}\subset \cdots\subset \{x_{\tau(1)},\ldots,x_{\tau(k)}\}\big].\]
Let $\PB_n(\O)^{(n-2)}$ be the $(n-2)$-skeleton of $\PB_n(\O)$ and let $\PP\coloneq \Poset(\PB_n(\O)^{(n-2)})$ be the poset of simplices of $\PB_n(\O)^{(n-2)}$. 
In the introduction we considered $\Tits_n(K)$ to be a simplicial complex, but let us define $\Tits_n(K)$ to be the poset of proper nonzero $K$-subspaces of $K^n$; the simplicial complex of flags from the introduction is then its geometric realization. Define the poset map $F\colon \PP\to \Tits_n(K)$ by \[F(\{v_1,\ldots,v_r\})=\spn_K(v_1,\ldots,v_r).\]
Since $\spn_K(v_1,\ldots,v_r)$ is an $r$-dimensional subspace of $K^n$, the restriction to simplices of $\PB_n(\O)^{(n-2)}$ (i.e.\ to simplices satisfying $0<r<n$) guarantees that this is a proper nonzero subspace, so $F$ is well-defined and also height-preserving.

Our claim is that the integral apartment class map $\phi\colon C_{n-1}\to \St_n(K)$ factors as the composition
\begin{equation}
\label{eq:intapartmentcomposition}
\begin{aligned}
\phi\colon C_{n-1}\xrightarrow{\partial} \ker(\partial\colon C_{n-2}&\to C_{n-3})= \widetilde{\HH}_{n-2}(\PB_n(\O)^{(n-2)};\Z)\\
&\xrightarrow[\iso]{b} \widetilde{\HH}_{n-2}(\PP;\Z)\xrightarrow{F_{\ast}} \widetilde{\HH}_{n-2}(\Tits_n(K);\Z)=\St_n(K).
\end{aligned}
\end{equation}
Indeed, consider an ordered basis $\vv=(v_1,\ldots,v_n)$ for $\O^n$. Under the boundary map 
$\partial\colon C_{n-1}\to C_{n-2}$, we have $\partial[\vv]=\sum_{i=1}^n (-1)^i [\vv^{(i)}]$, where $\vv^{(i)}=\{v_1,\ldots,\widehat{v}_i,\ldots,v_n\}$. Under $b$, this is taken to \[b(\partial [v])=\sum_{\sigma\in S_n}(-1)^\sigma \big[\{x_{\sigma(1)}\}\subset \cdots\subset \{x_{\sigma(1)},\ldots,x_{\sigma(n-1)}\}\big].\] 
Under $F_{\ast}$, the term corresponding to $\sigma\in S_n$ is taken to $\big[L_{\sigma(1)}\subset L_{\sigma(1)}+L_{\sigma(2)}\subset \cdots\subset L_{\sigma(1)}+\cdots + L_{\sigma(n-1)}\big]$. These are precisely the $(n-2)$-simplices making up the apartment $A_{\vv}$ (compare with \S\ref{section:buildings}  and Example~\ref{ex:Titsbuilding} below), and so $F_{\ast}(b(\partial[v]))=[A_{\vv}]$, as claimed.

To prove that $\phi$ is surjective, we prove that the maps $\partial$ and $F_{\ast}$ occurring in~\eqref{eq:intapartmentcomposition} are surjective. We first establish that under our assumptions, $\PB_n(\O)$ is CM of dimension $n-1$. In the case when $\abs{S}>1$ and $S$ contains a non-complex place, this is the statement of Theorem~\ref{maintheorem:partialbasescm}. In the remaining case when $\O$ is Euclidean, this was proved by Maazen~\cite{MaazenThesis} (see
\cite[Appendix]{LooijengaVanDerKallen} and \cite[Proof of Theorem B]{DayPutmanComplex}
for published proofs). 
By definition, the cokernel of $\partial\colon C_{n-1}\to \ker(\partial\colon C_{n-2}\to C_{n-3})$ is $\widetilde{\HH}_{n-2}(\PB_n(\O);\Z)$. We have just established that $\PB_n(\O)$ is CM of dimension $n-1$, and thus certainly $(n-2)$-acyclic. Therefore $\widetilde{\HH}_{n-2}(\PB_n(\O);\Z)=0$ and $\partial\colon C_{n-1}\to \ker(\partial\colon C_{n-2}\to C_{n-3})$ is surjective.

To show that $F_{\ast}$ is surjective, we will apply Theorem~\ref{theorem:quillen} to the height-preserving poset map $F$, so we must
verify that its hypotheses are satisfied. The
Solomon--Tits theorem says that $\Tits_n(K)$ is CM of dimension
$(n-2)$ (see e.g.\ \cite[Example~8.2]{QuillenPoset}). Given  $V \in \Tits_n(K)$, let $\ell = \dim(V)$, so $\height(V) = \ell-1$.  Lemma~\ref{lemma:summandsubspace}(c) 
implies that $V \cap \O^n$ is a rank $\ell$ direct summand of $\O^n$. Since $\abs{\class(\O)} = 1$, every f.g.\ projective $\O$-module is free, so $V \cap \O^n$ is isomorphic to $\O^\ell$. Lemma~\ref{lemma:partialbasissub}(b) states that $\{v_1,\ldots,v_k\}\in F_{\leq V}$ exactly if $\{v_1,\ldots,v_k\}$ is a partial basis of $V\cap \O^n$, so we have an isomorphism $F_{\leq V} \iso \Poset(\PartialBases_{\ell}(\O))$. We established in the previous paragraph that $\Poset(\PartialBases_{\ell}(\O))$ is CM of
dimension $\ell-1 = \height(V)$, so this verifies the remaining hypothesis of Theorem~\ref{theorem:quillen}. Applying Theorem~\ref{theorem:quillen}, we deduce that the map
\[F_{\ast}\colon \widetilde{\HH}_{n-2}(\PP;\Z) \rightarrow \widetilde{\HH}_{n-2}(\Tits_n(K);\Z) = \St_n(K)\]
is surjective. Since all other maps in~\eqref{eq:intapartmentcomposition} are isomorphisms, this shows that the integral apartment class map $\phi\colon C_{n-1}\to \St_n(K)$ is surjective.
\end{proof}

\section{Vanishing Theorem}
\label{section:vanishing}

In this section we deduce Theorem~\ref{maintheorem:vanishing}  from  Theorem~\ref{maintheorem:integralapartments}.  
\begin{proof}[Proof of Theorem~\ref{maintheorem:vanishing}]
In the special case $\lambda=0$, Theorem~\ref{maintheorem:vanishing}  states that 
$\HH^{\SLvcd}(\SL_n \O_K;K) = 0$ for all $n\geq 2$, and similarly for $\GL_n\O_K$. Since $\HH^{\ast}(\SL_n \O_K;K)\iso\HH^{\ast}(\SL_n \O_K;\Q)\tensor_\Q K$, the last claim of the theorem follows from the first ones, so those are all we need to prove.

Fix $\lambda = (\lambda_1,\ldots,\lambda_n)\in \Z^n$ with $\lambda_1\geq \cdots\geq \lambda_n$.
As we explained in \S\ref{section:topcohomologyintro}, Borel--Serre duality implies that $\St_n(K)$ is the rational dualizing module for $\SL_n\O_K$, which implies that
\[\HH^{\SLvcd}(\SL_n \O_K;V_{\lambda}) = (\St_n(K) \otimes V_{\lambda})_{\SL_n \O_K}.\]
Moreover, $\GL_n \O_K$ is an extension of $\SL_n \O_K$ by $\O_K^{\times}$.  By Dirichlet's unit theorem, $\O_K^\times\iso \Z^{r_1+r_2-1}\times \mu_K$, which is a virtual Poincar\'{e} duality group 
with virtual cohomological dimension $r_1+r_2-1$. It thus follows from Bieri--Eckmann~\cite[Theorem~3.5]{BieriEckmannDuality}  that $\GL_n \O_K$ is also a rational Bieri--Eckmann duality group with $\vcd(\GL_n \O_K)=\vcd(\SL_n \O_K)+\vcd(\O_K^\times)$ (as can be seen from the formulas in \eqref{eq:GLSLvcd}), and that the
rational dualizing module $D$ of $\GL_n\O_K$ satisfies
\begin{equation}
\label{eqn:glrestrict}
\text{Res}^{\GL_n \O_K}_{\SL_n \O_K} D \cong \St_n(K).
\end{equation}
We remark that it is not always the case that $D \cong \St_n(K)$ -- the actions of $\GL_n \O_K$ on $\St_n(K)$ and
$D$ sometimes differ by twisting by a certain character; see \cite{PutmanStudenmund} for a complete description.
In any case, \eqref{eqn:glrestrict} implies that
\[\HH^{\GLvcd}(\GL_n \O_K;V_{\lambda}) = (D \otimes V_{\lambda})_{\GL_n \O_K}\]
is a quotient of $(D \otimes V_{\lambda})_{\SL_n \O_K}\cong (\St_n(K) \otimes V_{\lambda})_{\SL_n \O_K}$.
The upshot is that to prove Theorem~\ref{maintheorem:vanishing}, it is enough to prove
that $(\St_n(K) \otimes V_{\lambda})_{\SL_n \O_K} = 0$ when $n\geq 2+\|\lambda\|$.

Define $\lambda' = (\lambda'_1,\ldots,\lambda'_{n-1},0)$ by $\lambda'_i = \lambda_i - \lambda_n$; observe that
$\|\lambda'\|=\|\lambda\|$.
As representations of $\SL_n \O_K$, the representations $V_{\lambda}$ and
$V_{\lambda'}$ are isomorphic; they only differ as representations of $\GL_n \O_K$.
Using Schur--Weyl duality, we can embed the $\SL_n K$-representation
$V_{\lambda'}$ as a direct summand of $(K^n)^{\otimes \ell}$ with $\ell \coloneq \|\lambda\|$.
It follows that it is enough to prove that
$(\St_n(K) \otimes (K^n)^{\otimes \ell})_{\SL_n \O_K} = 0$ when $n \geq 2+\ell$.

Fix an integral frame $\LL = \{L_1,\ldots,L_n\}$ with integral apartment class $[A_\LL]\in \St_n(K)$. We next show that $[A_{\LL}] \otimes (K^n)^{\otimes \ell}$ vanishes in the quotient $(\St_n(K) \otimes (K^n)^{\otimes \ell})_{\SL_n \O_K}$ if $n\geq 2+\ell$.
For each
$1 \leq i \leq n$, the intersection $L_i\cap \O_K^n$ is a rank~1 projective $\O_K$-module. Since $\abs{\cl(\O_K)}=1$ this module is free, so we can choose a generator $v_i$ for each $L_i\cap \O_K^n$. The assumption that $\LL$ is an integral frame means precisely that $\{v_1,\ldots,v_n\}$ is an $\O_K$-basis for $\O_K^n$. In particular $\{v_1,\ldots,v_n\}$ is a $K$-basis for $K^n$, so
$S = \Set{$v_{i_1} \otimes \cdots \otimes v_{i_{\ell}}$}{$1 \leq i_j \leq n$}$
is a basis for $(K^n)^{\otimes \ell}$.  

Consider an arbitrary basis vector $s = v_{i_1} \otimes \cdots \otimes v_{i_{\ell}} \in S$. Under the assumption $n\geq 2+\ell$,  we can find $1 \leq j < j' \leq n$ which
are distinct from all of the $i_1,\ldots,i_{\ell}$.  There exists a unique $g \in \SL_n \O_K$ such that 
\[g(v_i) = \begin{cases}
v_{j'} & \text{if $i = j$},\\
-v_{j} & \text{if $i = j'$},\\
v_i    & \text{otherwise}.\end{cases}\]
We have $g(s) = s$ and $g(\LL)=\LL$ by construction.
However $g$ exchanges the two lines $L_j$ and $L_{j'}$, so it reverses the orientation of the apartment $A_{\LL}$, i.e.\ $g([A_{\LL}]) = -[A_{\LL}]$. This means that $[A_{\LL}] \otimes s$ and $-[A_{\LL}] \otimes s$ become equal, and thus vanish, in the coinvariant quotient $(\St_n(K) \otimes (K^n)^{\otimes \ell})_{\SL_n \O_K}$. Since $s$ was arbitrary, $[A_{\LL}] \otimes (K^n)^{\otimes \ell}$ vanishes in this quotient for every integral apartment $A_{\LL}$. Theorem~\ref{maintheorem:integralapartments} states that $\St_n(K)$ is generated
by integral apartment classes, and so we conclude that $(\St_n(K) \otimes (K^n)^{\otimes \ell})_{\SL_n \O_K}=0$, as desired.
\end{proof}

\begin{remark}
\label{rem:thmCfalseimaginary}
Some condition on $K$ beyond $\abs{\class(\O_K)} = 1$ is necessary in Theorem
\ref{maintheorem:vanishing}.  Indeed, for $d<0$ squarefree let
$\O_d$ denote the ring of integers in the quadratic imaginary field $K_d=\Q(\sqrt{d})$.  Those $d < 0$ for which
$\O_d$ is non-Euclidean but satisfies $\abs{\class(\O_d)}=1$ are exactly
$d\in \{-19, -43, -67, -163\}$.  Although $\HH^2(\SL_2 \O_{-19};\Q) = 0$, we have
\[\HH^2(\SL_2 (\O_{-43});\Q) = \Q,\quad \HH^2(\SL_2 (\O_{-67});\Q) = \Q^2,\quad \HH^2(\SL_2 (\O_{-163});\Q) = \Q^6.\]
For these calculations see Vogtmann~\cite{VogtmannBianchi} and Rahm~\cite{RahmTorsion}.  Presumably,
similar things occur for $\SL_n \O_{d}$ for $n \geq 3$, but we could not find such calculations 
in the literature.  It is likely that these four fields $\Q(\sqrt{-19})$, $\Q(\sqrt{-43})$, $\Q(\sqrt{-67})$, and $\Q(\sqrt{-163})$ provide the
only such counterexamples to the conclusion of Theorem~\ref{maintheorem:vanishing}; indeed, 
Weinberger~\cite{WeinbergerEuclidean}
proved that the generalized Riemann hypothesis implies that if $\O_K$ has class number $1$
and infinitely many units,
then $\O_K$ is Euclidean.
\end{remark}

\section{Non-Vanishing and Non-Integrality Theorems}
\label{section:prelimnonvanishing}

In this section we prove Theorem~\ref{maintheorem:nonintegrality} and Theorem~\ref{maintheorem:nonvanishing}; in fact we prove stronger versions of these results in Theorems~\ref{thm:nonintegralitystronger} and~\ref{thm:nonvanishingstronger}, which apply to projective modules $M$ that need not be free.  After summarizing in \S\ref{section:buildings} some basic properties of spherical buildings, in \S\ref{section:structureTitsM} we prove Proposition~\ref{prop:foldedapartment}, the key technical theorem 
of this section.  We apply this proposition in \S\ref{section:nonvanishingnonintegrality} to prove Theorems~\ref{thm:nonintegralitystronger} and~\ref{thm:nonvanishingstronger}.

\subsection{Spherical buildings}
\label{section:buildings}

In this section we briefly review the classical theory of spherical buildings.  See  \cite{BrownBuildings} for a general reference.

A $d$-dimensional spherical building $Y$ is a simplicial complex satisfying certain axioms.  
Associated to $Y$ is a finite Coxeter group $W$ of rank $d+1$ called the
\emph{Weyl group} and a canonical collection 
of subcomplexes $A\subset Y$ called \emph{apartments}.  Each 
apartment is isomorphic to the 
Coxeter complex of $W$ and is thus acted upon by $W$.  
In particular, each apartment $A$ is homeomorphic to a  sphere 
$S^d$ and defines a fundamental class $[A]\in \widetilde{\HH}_d(Y)$. 
The class $[A]$ is only defined up to $\pm1$, but for our purposes this ambiguity will never pose a problem, so we do not mention it again.  
A $d$-simplex $C$ in an apartment $A$ is called a \emph{chamber} of $A$.   The apartment $A$ is 
the union of the translates $wC$ for $w\in W$, and 
$[A]= \sum_{w\in W}(-1)^w [wC]$. 

Any $d$-dimensional spherical building $Y$ is CM of dimension $d$. 
The top homology $\widetilde{\HH}_d(Y)$ is 
spanned by the fundamental classes $[A]$ of apartments; indeed, for any fixed chamber $C$, a basis for $\widetilde{\HH}_d(Y)$ is given by the apartment classes $[A]$ of all apartments $A$ containing $C$.

The following three examples will be of the most interest to us. A poset $P$ is said to be a $d$-dimensional spherical building if its realization $\abs{P}$ is.

\begin{example}[{\bf Poset of subspaces}]
\label{ex:Titsbuilding}
Given an $n$-dimensional vector space $W$ over a field $K$, let $\Tits(W)$ be the poset of proper, nonzero $K$-subspaces of $W$. Its realization is the spherical Tits building associated to $\GL(W)$, an $(n-2)$-dimensional spherical building with Weyl group $S_n$. 
The chambers of $\Tits(W)$ are the complete flags
\[0 \subsetneq W_1 \subsetneq \cdots \subsetneq W_{n-1} \subsetneq W.\]
The Coxeter complex of $S_n$ can be identified with 
the barycentric subdivision of $\partial \Delta^{n-1}$.  As 
discussed in the introduction, apartments of $\Tits(W)$ are in bijection 
with \emph{frames} of $W$, i.e.\ sets of lines $\LL=\{L_1,\ldots,L_n\}$ 
such that $W=L_1\oplus \cdots\oplus L_n$.  The apartment $A_{\LL}$ corresponding
to $\LL$ satisfies $[A_{\LL}] = \sum_{\sigma \in S_n} (-1)^{\sigma} \VV_{\sigma}$,
where $\VV_{\sigma}$ is the chamber
\begin{equation}
\label{eq:Titschambers}
L_{\sigma(1)} \subsetneq  L_{\sigma(1)}+L_{\sigma(2)} \subsetneq
\cdots \subsetneq L_{\sigma(1)}+\cdots+L_{\sigma(n-1)}.
\end{equation}
\end{example}

\begin{example}[{\bf Poset of summands}]
\label{ex:Titssummands}
Given a finite rank projective $\O$-module $M$ over a Dedekind domain $\O$, let $\Tits(M)$ denote the poset of proper nonzero $\O$-submodules $U\subset M$ which are direct summands of $M$. 
Recall from Lemma~\ref{lemma:summandsubspace}(c) that the assignment $V\mapsto V\cap M$ gives a bijection between the $K$-subspaces $V \subsetneq M\tensor_\O K$ and the direct summands of the $\O$-module $M$. This assignment preserves containment, so it gives a canonical isomorphism $\Tits(M\tensor_\O K)\iso \Tits(M)$. In particular, if $M$ has rank $n$, then $\Tits(M)$ is an $(n-2)$-dimensional spherical building.
The apartments of $\Tits(M)$ correspond to frames $\LL=\{L_1,\ldots,L_n\}$ of $M\tensor_\O K$, or equivalently to sets $\II=\{I_1,\ldots,I_n\}$ of rank $1$ summands of $M$ such that $I_1+\cdots+I_n$ has rank $n$.  The formula~\eqref{eq:Titschambers} might suggest that the chambers of $A_{\II}$ are given by e.g.\ $I_1\subsetneq I_1+I_2\subsetneq \cdots$, but this is  not the case: indeed $I_1+I_2$ need not be a summand of $M$ at all! The corresponding summand is  $M\cap ((I_1+I_2)\otimes_\O K)$, the ``saturation'' of $I_1+I_2$. For $n\in \N$, let $[n]= \{1,\ldots,n\}$, and given $X\subset[n]$ set \begin{equation}
\label{eq:UX}
U_X \coloneq  M\cap (\sum_{x\in X} I_x\otimes_\O K).
\end{equation}
 The chambers of $A_{\II}$ are then
\[\VV_\sigma\colon\quad U_{\sigma([1])}\subsetneq U_{\sigma([2])}\subsetneq \cdots\subsetneq U_{\sigma([n-1])}\qquad\quad\text{for } \sigma\in S_n,\]
and the apartment class
$[A_{\II}]\in \widetilde{\HH}_{n-2}(\Tits(M);\Z)$ is given by $[A_{\II}]=\sum_{\sigma\in S_n}(-1)^\sigma \VV_\sigma$.
A frame $\II=\{I_1,\ldots,I_n\}$ is \emph{$M$-integral} if $I_1+\cdots+I_n=M$. In this case we have $M=I_1\oplus \cdots\oplus I_n$ and $U_X=\bigoplus_{x\in X}I_X$, and the chambers of the \emph{$M$-integral apartment} $A_{\II}$ are indeed given by 
\[I_{\sigma(1)} \subsetneq  I_{\sigma(1)} \oplus I_{\sigma(2)} \subsetneq
\cdots \subsetneq I_{\sigma(1)}\oplus\cdots\oplus I_{\sigma(n-1)}.
\]
\end{example}

\begin{example}[{\bf Join of discrete posets}]
\label{ex:joindiscrete}
Let $T$ be a finite set.  Define a poset $X_m(T) = \{1,\ldots,m\} \times T$ by declaring $(p,t) < (p',t')$ exactly when
$p < p'$.  The poset $X_m(T)$ is an $(m-1)$-dimensional spherical building with Coxeter group
$W = (\Z/2)^{m}$.  The chambers of $X_m(T)$ are chains
\[(1,t_1) < (2,t_2) < \ldots < (m,t_{m})\]
with $t_i \in T$ for $1 \leq i \leq m$.  We will denote this chamber by
$(t_1,\ldots,t_m)$.
The apartments of $X_m(T)$ are in bijection with ordered sequences
$\bS = (S_1,\ldots,S_{m})$, where $S_i = \{a_i,b_i\}$ is a 
$2$-element subset of $T$ with $a_i \neq b_i$.
For $\e=(\epsilon_1,\ldots,\epsilon_{m})\in (\Z/2)^{m}$,
let $c^\e_k=b_k$ if $\epsilon_k=0$ and $c^\e_k=a_k$ if $\epsilon_k=1$ (this
 convention will simplify some of our later formulas).  Denote the
corresponding chamber by $\CC_\e=(c^\e_1,\ldots,c^\e_{m})$.  The apartment
$A_{\bS}$ is the union of the $2^{m}$ chambers $\CC_\e$, and its fundamental
class is
\[[A_{\bS}]=\sum_\e (-1)^{\e} [\CC_\e] \quad \quad \text{where $(-1)^{\e} = (-1)^{\sum_i \epsilon_i}$}.\]
$X_m(T)$ is the $m$-fold join of the discrete set $T$, so
it is homotopy equivalent to a wedge of $(\abs{T}-1)^{m}$ copies of $S^{m-1}$. In particular, $\widetilde{\HH}_{m-1}(X_m(T);\Z)\iso \Z^N$ with $N={(\abs{T}-1)^m}$.
\end{example}

\subsection{Cycles in \texorpdfstring{$\Tits(M)$}{T(M)} and the class group \texorpdfstring{$\cl(\O)$}{clO} }
\label{section:structureTitsM}

Let $M$ be a rank $n$ projective $\O$-module and let $\T(M)$ be the Tits building discussed
in Example~\ref{ex:Titssummands}.  The space $\abs{\T(M)}$ is a wedge of $(n-2)$-dimensional spheres;
define $\St(M) = \widetilde{\HH}_{n-2}(\abs{\T(M)};\Z)$.  As we discussed in Example~\ref{ex:Titssummands}, we have
\begin{equation}
\label{eqn:identifysteinberg}
\St(M) \cong \St(M \otimes_O K) \cong \St_n(K).
\end{equation}
Our goal in this section is to relate $\St(M)$ to the homology of the 
much simpler complex $X_{n-1}(\cl(\O))$ that we introduced in Example~\ref{ex:joindiscrete}.  Our main theorem
will imply that there is a $\GL(M)$-invariant surjection
\[\St(M) \twoheadrightarrow \widetilde{\HH}_{n-2}(\abs{X_{n-1}(\cl(\O))};\Z).\]
Define a map
$\psi\colon \T(M) \rightarrow X_{n-1}(\cl(\O))$ via the formula
$\psi(U)=\big(\Rank(U),[U]\,\big)$.
Given $U,U' \in \Tits(M)$ with $U\subsetneq U'$, Lemma~\ref{lemma:summandsubspace}(b) says that $U$ is a summand of $U'$,
so $\Rank(U)<\Rank(U')$.  Thus $\psi$ is a poset map.
Elements $g \in \GL(M)$ and $U \in \Tits(M)$ determine an isomorphism $g\colon U\iso g(U)$, so $\psi(g(U)) = \psi(U)$. Therefore $\psi$ induces
a map $\overline{\psi}\colon \abs{\Tits(M)} / \GL(M) \rightarrow \abs{X_{n-1}(\class(\O))}$.

\begin{proposition}
\label{proposition:quotientdesc}
Let $M$ be a rank $n$ projective $\O$-module.
The map $\overline{\psi}\colon \abs{\Tits(M)} / \GL(M) \rightarrow \abs{X_{n-1}(\class(\O))}$ defined
above is then an isomorphism of cell complexes.
\end{proposition}
\begin{proof}
The proposition is equivalent to the following two claims; the first implies that $\overline{\psi}$ is surjective
and the second implies that $\overline{\psi}$ is injective.
\begin{enumerate}[nosep]
\item Let $(j_1,c_1) \lneq \cdots \lneq (j_p,c_{p})$ be a chain in $X_{n-1}(\class(\O))$.  Then
there exists a chain $U_1 \subsetneq \cdots \subsetneq U_p$ in $\Tits(M)$ such that
$\psi(U_i) = (j_i,c_i)$ for all $1 \leq i \leq p$.
\item For $\ell=1,2$ let $V_1^{\ell} \subsetneq \cdots \subsetneq V_q^{\ell}$ be a chain in $\Tits(M)$.  Assume
that $\psi(V_i^{1}) = \psi(V_i^2)$ for all $1 \leq i \leq q$.  Then there exists some $g \in \GL(M)$
such that $g(V_i^1) = V_i^2$ for all $1 \leq i \leq q$.
\end{enumerate}
We begin with the first claim.  Set 
$(j_0,c_0) = (0,0)$ and $(j_{p+1},c_{p+1}) = (n-j_p,[M])$.
For $1 \leq i \leq p+1$, let $\widehat{U}_i$ be a projective $\O$-module
of rank $j_i-j_{i-1}$ with $[\widehat{U}_i] = c_i - c_{i-1}$.
Then $\widehat{U}_1 \oplus \cdots \oplus \widehat{U}_{p+1}$ is a projective $\O$-module of rank $n$ satisfying
\[[\widehat{U}_1 \oplus \cdots \oplus \widehat{U}_{p+1}] = c_1 + (c_2 - c_1) + \cdots + (c_p - c_{p-1}) + ([M] - c_p) = [M],\]
so there exists an isomorphism $\eta\colon \widehat{U}_1 \oplus \cdots \oplus \widehat{U}_{p+1} \rightarrow M$.
The desired summands $U_i$ are then
$U_i\coloneq  \eta(\widehat{U}_1) \oplus \cdots \oplus \eta(\widehat{U}_i)$.

We now turn to the second claim.  For both $\ell=1$ and $\ell=2$, set $V_0^{\ell} = 0$ and $V_{q+1}^{\ell} = M$.
For $1 \leq i \leq q+1$, Lemma~\ref{lemma:summandsubspace}(b) implies that $V_{i-1}^{\ell}$ is a summand
of $V_i^{\ell}$, so there exists some $\widehat{V}_i^{\ell} \subset V_i^{\ell}$ such that
$V_i^{\ell} = V_{i-1}^{\ell} \oplus \widehat{V}_i^{\ell}$.  We thus have
$V_i^{\ell} = \widehat{V}_1^{\ell} \oplus \cdots \oplus \widehat{V}_i^{\ell}$.  Observe that
\[[\widehat{V}_i^{1}] = [V_i^{1}] - [V_{i-1}^{1}] = c_i - c_{i-1} = [V_i^2] - [V_{i-1}^2] = [\widehat{V}_i^{2}].\]
and similarly $\Rank(\widehat{V}_i^1)=\Rank(\widehat{V}_i^2)$.  Therefore there exists an isomorphism $\zeta_i\colon \widehat{V}_i^{1} \rightarrow \widehat{V}_i^{2}$.
The desired $g \in \GL(M)$ is then the composition
\[M = \widehat{V}_1^{1} \oplus \cdots \oplus \widehat{V}_{q+1}^1 \stackrel{\zeta_1 \oplus \cdots \oplus \zeta_{q+1}}{\longrightarrow} \widehat{V}_1^{2} \oplus \cdots \oplus \widehat{V}_{q+1}^2 = M. \qedhere\]
\end{proof}

We now turn to the main technical theorem of \S\ref{section:prelimnonvanishing}, which asserts that the map
\[\psi_{\ast}\colon \widetilde{\HH}_{n-2}(\Tits(M);\Z) \rightarrow \widetilde{\HH}_{n-2}(X_{n-1}(\class(\O));\Z)\] 
is surjective.  We will use the notation for apartments and chambers in these buildings that
was introduced in Examples~\ref{ex:Titssummands} and~\ref{ex:joindiscrete}.
The reason that this result is so hard to prove is that for a general  apartment $A_{\II}$
in $\Tits(M\tensor_\O K)$, it seems quite difficult to describe the image $\psi_{\ast}([A_{\II}])$ of its fundamental class in the homology of $X_{n-1}(\class(\O))$.
For one thing, we saw in Example~\ref{ex:Titssummands} that the summands $U_X$ that make up the chambers of $A_{\II}$ are not easily understood; for example, knowing the images $\psi(I_k)$ of the rank~1 summands does not allow us to determine $\psi(U_X)$. For another, the apartment $A_\II$
consists of $n!$ chambers, while the apartments in $X_{n-1}(\class(\O))$ are made up of
only $2^{n-1}$ chambers.  Since $n!\gg 2^{n-1}$, in general we expect that $\psi(A_\II)$
will be draped over many apartments in $X_{n-1}(\class(\O))$.   However, the following
proposition guarantees that we can find certain special apartments $A_\II$ whose
image is just a single apartment $A_{\bS}$, with the excess $n!-2^{n-1}$ chambers
all ``folded away'' and canceling each other out precisely, leaving only the desired homology class.

\begin{proposition}[{\bf Detecting spaces of cycles}]
\label{prop:foldedapartment}
Let $\O$ be a Dedekind domain, let $M$ be a rank $n$ projective $\O$-module, and
let $A_{\bS}$ be an apartment of $X_{n-1}(\class(\O))$.  Then there exists a frame
$\II$ of $M$ such that
$\psi_{\ast}([A_\II])=[A_{\bS}] \in \widetilde{\HH}_{n-2}(X_{n-1}(\class(\O));\Z)$.
In particular, 
$\psi_\ast\colon \widetilde{\HH}_{n-2}(\Tits(M);\Z)\to \widetilde{\HH}_{n-2}(X_{n-1}(\class(\O));\Z)$ is surjective.
\end{proposition}
\begin{proof}
Write $\bS=(S_1,\ldots,S_{n-1})$ with 
$S_i=\{a_i,b_i\}$ for some $a_i,b_i\in\class(\O)$.  For $\e=(\epsilon_1,\ldots,\epsilon_{n-1})\in (\Z/2)^{n-1}$, 
let $c^\e_k=b_k$ if $\epsilon_k=0$, and $c^\e_k=a_k$ if $\epsilon_k=1$ (as
we noted in Example~\ref{ex:joindiscrete}, this convention will simplify some of our formulas).  The apartment 
$A_{\bS}$ is the union of the $2^{n-1}$ chambers $\CC_\e=(c^\e_1,\ldots,c^\e_{n-1})$, and its fundamental 
class is 
\[[A_{\bS}]=\sum_\e (-1)^{\e} [\CC_\e] \quad \quad \text{with $(-1)^{\e} = (-1)^{\sum_i \epsilon_i}$}.\]

Using Proposition~\ref{proposition:quotientdesc} we may choose summands $B_k$ for $1 \leq k \leq n-1$ and $A_k^{(i)}$ for $1 \leq i \leq k \leq n-1$ with the following properties:
\begin{align*}
0=B_0&\subsetneq B_1 \subsetneq \cdots \subsetneq B_{n-1} \subsetneq B_n=M& \Rank(B_k)&=k,[B_k]=b_k.\\
B_{i-1}& \subsetneq A_i^{(i)} \subsetneq A_{i+1}^{(i)} \subsetneq \cdots \subsetneq A_{n-1}^{(i)} \subsetneq M& \Rank(A_k^{(i)})& = k, [A_k^{(i)}] = a_k, A_k^{(i)} \subset B_{k+1}
\end{align*}
It may not be immediately obvious how these should be chosen, so we take the time to explain carefully. First, choose all the $B_k$, which we may do by Proposition~\ref{proposition:quotientdesc}. Then for each $i$, we will choose the $A_k^{(i)}$ inductively in increasing order of $k$ to satisfy $A_{k-1}^{(i)}\subsetneq A_k^{(i)}\subsetneq B_{k+1}$ (or $B_{i-1}\subsetneq A_i^{(i)}\subsetneq B_{i+1}$ in the base case when $k=i$). To choose each $A_k^{(i)}$ thus requires solving the following problem: given $k$, given $y\in \cl(\O)$, and given fixed summands $X\subsetneq Z$ with $\Rank(X)=k-1$ and $\Rank(Z)=k+1$, find a rank $k$ summand $Y$ with $X\subsetneq Y\subsetneq Z$ and $[Y]=y$. By Lemma~\ref{lemma:summandsubspace}(b), $X$ is a summand of $Z$, so write $Z=X\oplus U$. Applying Proposition~\ref{proposition:quotientdesc} to $U$, we can find a rank 1 summand $L\subset U$ with $[L]=y-[X]$; setting $Y=X\oplus L$ gives the desired solution.

We now define
\[I_1\coloneq B_1 \quad \text{and} \quad I_k \coloneq \bigcap_{i=1}^{k-1} A_{k-1}^{(i)} \quad \text{for $1<k\leq n$}.\]
Set $\II = \{I_1,\ldots,I_n\}$.  We will prove that $\II$ is a frame of $M$ and that
$\psi_{\ast}([A_{\II}]) = [A_{\bS}]$.

\BeginClaims
\begin{claims}
\label{claim:frame}
For each $1\leq k\leq n$ the intersection $I_k$ is a rank~1 summand of $M$.  Also, $I_1 + \cdots + I_k$ is a rank~$k$ $\O$-submodule
of $M$ (not necessarily a summand).  In particular, the rank of $I_1 + \cdots I_n$ is $n$, so
$\II$ is a frame of $M$.
\end{claims}
\begin{proof}[Proof of Claim~\ref{claim:frame}]
The bijection $V\mapsto M\cap V$ of Lemma~\ref{lemma:summandsubspace}(c) preserves intersections, so 
the intersection of summands of $M$ is also a summand of $M$.  In particular, each $I_k$ is a summand of $M$ 
and $I_k\tensor_\O K=\bigcap_{i=1}^{k-1} A_{k-1}^{(i)} \tensor_\O K$.  The intersection of the $k-1$ 
hyperplanes $A_{k-1}^{(1)} \tensor_\O K, \ldots, A_{k-1}^{(k-1)} \tensor_{\O} K$ 
inside the $k$-dimensional vector space $B_k\tensor_\O K$ cannot vanish, so $\Rank(I_k) \geq 1$. We will prove equality holds, but not until the end of the proof.

Since each $A_{k-1}^{(i)}$ is contained in $B_k$, we have $I_k\subset B_k$.  Since $B_j\subset B_k$ for $j<k$, 
we see that $I_1+\cdots+I_k \subset B_k$ for all $k$.

We next claim that $I_k\cap B_{k-1}=0$ for all $k$.   The proof will be by induction on $k$.  The base case
$k=1$ is trivial since $I_1=B_1$, so assume that $k > 1$.  For $1 \leq i \leq k-2$, we have
$[A_{k-1}^{(i)}] = a_{k-1} \neq b_{k-1} = [B_{k-1}]$, so $A_{k-1}^{(i)}\neq B_{k-1}$. 
It follows that $A_{k-1}^{(i)}\cap B_{k-1}$ is a summand of rank at most $k-2$.  But by construction, the rank $k-2$ summand
$A_{k-2}^{(i)}$ is contained in both $A_{k-1}^{(i)}$ and in $B_{k-1}$, so it must be their intersection: 
$A_{k-1}^{(i)} \cap B_{k-1}=A_{k-2}^{(i)}$ for $1 \leq i \leq k-2$.  
In a similar way, we have $A_{k-1}^{(k-1)} \cap B_{k-1} = B_{k-2}$. Combining these, we find that
\begin{align*}
I_{k} \cap B_{k-1}&=\big(\bigcap_{i=1}^{k-1} A_{k-1}^{(i)}\big)\cap B_{k-1}=\bigcap_{i=1}^{k-1}(A_{k-1}^{(i)}\cap B_{k-1})\\
&=\big(\bigcap_{i=1}^{k-2}A_{k-2}^{(i)}\big)\cap B_{k-2}=I_{k-1}\cap B_{k-2}.
\end{align*}
By induction, this is zero, as claimed.

Finally, we prove $I_k$ has rank $1$ and that $I_1 + \cdots + I_k$ has rank $k$.  The proof of this
is by induction on $k$.  The base case $k=1$ is trivial, so assume that $k>1$, that $I_{k-1}$ has rank $1$,
and that $I_1 + \cdots + I_{k-1}$ has rank $k-1$.
From the containments
\[I_1+\cdots+I_{k-1}\subset I_1+\cdots+I_k\subset B_k\]
we see that $k-1\leq \Rank(I_1+\cdots+I_k)\leq k$. Assume for the sake of contradiction that $\Rank(I_1+\cdots+I_k)<k$ 
or that $\Rank(I_k)>1$.  In either case, the subspaces  $I_k\tensor_\O K$ and $(I_1+\cdots+I_{k-1})\tensor_\O K$  must intersect nontrivially inside the $k$-dimensional vector space $B_k\tensor_\O K$. Multiplying by an appropriate denominator, it follows that $(I_1+\cdots+I_{k-1})\cap I_k\neq 0$. However we proved above that $B_{k-1}\cap I_k=0$; since $I_1+\cdots+I_{k-1}\subset B_{k-1}$, this is a contradiction. Therefore $\Rank(I_1+\cdots+I_k)=k$ and $\Rank(I_k)=1$, as desired. This concludes the proof of Claim~\ref{claim:frame}.
\end{proof}

It remains to prove that $\psi_{\ast}([A_{\II}]) = [A_{\bS}]$.  As explained in Example~\ref{ex:Titssummands},
the chambers of the apartment $A_\II$ are given by
\[\VV_\sigma\coloneq\quad U_{\sigma([1])}\lneq U_{\sigma([2])}\lneq\cdots\lneq U_{\sigma([n-1])}\qquad\text{ for }\sigma\in S_n.\]
The image $\psi(\VV_\sigma)\subset \psi(A_\LL)\subset X_{n-1}(\class(\O))$ is the 
chamber determined by the sequence
\[\big(\,[U_{\sigma([1])}],\,[U_{\sigma([2])}],\ldots,[U_{\sigma([n-1])}]\,\big).\]
For most subsets $X\subset [n]$, we cannot say anything about the summand $U_X$ or its 
class $[U_X]$, but there are two exceptions.
\begin{claims}[Trichotomy]
\label{claim:trichotomy}\ \change

\begin{enumerate}[label={\normalfont \arabic*.},topsep=3pt,itemsep=2pt,parsep=2pt]
\item $U_{[k]} = B_k$, so $[U_{[k]}] = b_k$.  
\item If $X = [k+1] \setminus \{i\}$ for some $1 \leq i \leq k$, then
$U_X = A_k^{(i)}$, so $[U_X]=a_k$.
\item For other subsets $X\subset [n]$, we know nothing.
\end{enumerate}
\end{claims}
\begin{proof}[Proof of Claim~\ref{claim:trichotomy}]
Case 1: We saw above that $I_1+\cdots+I_k\subset B_k$, and Claim~\ref{claim:frame} states that $\Rank(I_1+\cdots+I_k)=k$, so $(I_1+\cdots+I_k)\tensor_\O K=B_k\tensor_\O K$. Since $U_{[k]}=M\cap ((I_1+\cdots+I_k)\tensor_\O K)$ by definition, this means that $U_{[k]}=M\cap (B_k\tensor_\O K)$. But $B_k$ is a summand of $M$, so $M\cap (B_k\tensor_\O K)=B_k$.
 
Case 2: We begin by observing that $I_x\subset A_k^{(i)}$ whenever $x\leq k+1$ and $i\neq x$.  For $i<x$ this follows 
from  $I_x\subset A_{x-1}^{(i)}$ (by the definition of $I_x$) and 
$A_{x-1}^{(i)}\subset A_k^{(i)}$.  For $i> x$, we use instead that 
$I_x\subset B_x\subset B_{i-1}\subset A_i^{(i)}\subset A_k^{(i)}$.  
Let $X=[k+1]\setminus\{i\}$ for some $1\leq i\leq k$. It follows from Claim~\ref{claim:frame} that $I_1+\cdots+\widehat{I}_i+\cdots+I_{k+1}$ has rank $k$, and the beginning of this paragraph shows that this submodule is contained in the rank $k$ summand $A_k^{(i)}$. Just as above, this implies that $U_X=M\cap((I_1+\cdots+\widehat{I}_i+\cdots+I_{k+1})\tensor_\O K)$ coincides with $A_k^{(i)}$. This concludes the proof of Claim~\ref{claim:trichotomy}.
\end{proof}

We can analyze this trichotomy in terms of $\sigma\in S_n$, as follows.  Given 
$\sigma\in S_n$, think of $\sigma$ as a function $\sigma\colon [n]\to [n]$, and 
define $s_\sigma\colon [n]\to [n]$ by 
\[s_\sigma(k)=\sup\Set{$\sigma(i)$}{$1\leq i\leq k$}.\]
The function $s_\sigma$ is monotone increasing and satisfies $k\leq s_\sigma(k)\leq n$.
If $s_\sigma(k)=k$ then it must be that $\sigma([k])=[k]$, so $[U_{\sigma([k])}]=[B_k]=b_k$. 
If $s_\sigma(k)=k+1$, it must be that $\sigma([k])=[k+1]-\{i\}$ for some $1\leq i\leq k$, so 
$[U_{\sigma([k])}]=[A_k^{(i)}]=a_k$.  To sum up, we have
\begin{equation}
\label{eq:clVsigma}
[U_{\sigma([k])}]=\begin{cases}
b_k&\text{if }s_\sigma(k)=k\\
a_k&\text{if }s_\sigma(k)=k+1\\
\text{unknown}&\text{otherwise}
\end{cases}
\end{equation}
Accordingly, we divide the permutations $\sigma\in S_n$ into two subsets: we say that 
$\sigma\in S_n$ is \emph{good} if $s_\sigma(k)\in \{k,k+1\}$ for all $k$, and say that 
$\sigma$ is \emph{bad} otherwise.  

\begin{claims}
\label{claim:gooddesc}
The good permutations are the $2^{n-1}$ permutations of the form
\[\sigma_{\e}\coloneq (1\ \,2)^{\epsilon_1}(2\ \,3)^{\epsilon_2}\cdots (n-1\ \,n)^{\epsilon_{n-1}} \in S_n\]
for some $\e = (\epsilon_1,\ldots,\epsilon_{n-1}) \in (\Z/2)^{n-1}$. For such a good permutation $\sigma=\sigma_\e$,
\[s_\sigma(k)=k\iff\epsilon_k=0\qquad s_\sigma(k)=k+1\iff\epsilon_k=1\qquad\text{for }k<n.\]
\end{claims}
\begin{proof}[Proof of Claim~\ref{claim:gooddesc}]
The proof is by induction on $n$.  The base case is $n=1$, where
the claim is trivial.  Assume now that $n\geq 2$ and that the claim is true
for smaller $n$.  Let $\sigma \in S_n$ be a good permutation.
Since $s_\sigma(k)\geq \sigma(k)$, the 
condition $s_\sigma(k)\leq k+1$ for all $k$ implies that $\sigma(k)\neq n$ for any $k<n-1$. 
Therefore we have either $\sigma(n)=n$ or $\sigma(n-1)=n$.  In the first case, 
the restriction $\sigma|_{[n-1]}$ is a good permutation in $S_{n-1}$, so the claim holds by induction with $\epsilon_{n-1}=0$; since $\sigma(n)=n$ in this case, we must have $s_\sigma(n-1)=n-1$. 
In the second case, we have $s_\sigma(n-1)=n$. Moreover, the permutation $\sigma'\coloneq \sigma (n-1\,\ n)$ satisfies 
$\sigma'(n)=n$ and is good (since $\sigma'|_{[n-2]}=\sigma|_{[n-2]}$), so applying the first case to $\sigma'$ proves Claim~\ref{claim:gooddesc}.
\end{proof}

We can now divide the chambers of $A_{\II}$ between the two chains defined by
\[[A_{\II}]^{\text{good}} \coloneq  \sum_{\sigma \in S_n \text{ good}} (-1)^{\sigma} \VV_{\sigma} \quad \text{and} \quad [A_{\II}]^{\text{bad}} \coloneq \sum_{\sigma \in S_n \text{ bad}} (-1)^{\sigma} \VV_{\sigma},\]
so $[A_{\II}] = [A_{\II}]^{\text{good}} + [A_{\II}]^{\text{bad}}$ as chains in $C_{n-2}(\Tits(M))$.
\begin{claims}
\label{claim:dealgood}
We have $\psi_{\ast}([A_{\II}]^{\text{good}}) = [A_{\bS}]$.
\end{claims}
\begin{proof}[Proof of Claim~\ref{claim:dealgood}]
Consider $\e = (\epsilon_1,\ldots,\epsilon_{n-1}) \in (\Z/2)^{n-1}$.
By Claim~\ref{claim:gooddesc}, $s_{\sigma_\e}(k)=k+\epsilon_k$ for all $1\leq k<n$.  By~\eqref{eq:clVsigma} 
this implies $[U_{\sigma_\e([k])}]=c^\e_k$ for all $k$ (recall that $c^\e_k$ is either 
$b_k$ or $a_k$ depending on whether $\epsilon_k$ is $0$ or $1$).  Therefore the image 
$\psi(\VV_{\sigma_\e})$ of the chamber $\VV_{\sigma_\e}$ is precisely the chamber 
$\CC_\e$.  Since 
\[(-1)^{\sigma_\e}=(-1)^{\sum \epsilon_i}=(-1)^\e,\]
the orientations of $\VV_{\sigma_\e}$ and $\CC_\e$ agree, proving Claim ~\ref{claim:dealgood}.\end{proof}

To deal with $[A_{\II}]^{\text{bad}}$, we 
define an involution $\sigma \mapsto \overline{\sigma}$ on the set of bad permutations 
$\sigma\in S_n$ as follows. Given a bad permutation $\sigma$, define
\begin{align*}
x_\sigma&=\sup\Set{$x\in \{2,\ldots,n-1\}$}{$s_\sigma(x-1)>x$},\\
y_\sigma&=\inf\Set{$y\in \{x_\sigma+1,\ldots,n\}$}{$s_\sigma(y)=y$}.
\end{align*}
There exists some $x$ with $s_\sigma(x-1)>x$ if and only if $\sigma$ is bad, so $x_\sigma$ 
is well-defined.  The second condition is always satisfied by $y=n$, so $y_\sigma$ is 
well-defined as well.  We define 
$\overline{\sigma}\coloneq \sigma (x_\sigma\ \,y_\sigma)\in S_n$.

\begin{claims}
\label{claim:invol}
The map $\sigma \mapsto \overline{\sigma}$ is an involution on the set
of bad permutations.
\end{claims}
\begin{proof}[Proof of Claim~\ref{claim:invol}]
Let $\sigma \in S_n$ be a bad permutation. We will show that $s_{\overline{\sigma}}=s_\sigma$; since $x_\sigma$ and $y_\sigma$ only depend on $s_\sigma$, 
this implies that  $x_{\overline{\sigma}}=x_\sigma$ and $y_{\overline{\sigma}}=y_\sigma$ and thus that $\overline{\overline{\sigma}} = \sigma$.

For $k<x_\sigma$ or $k\geq y_\sigma$ we have $\sigma([k])=\overline{\sigma}([k])$, so $s_{\overline{\sigma}}(k)=s_\sigma(k)$ automatically.
Considering the remaining terms we find that
\begin{equation}
\label{eq:xsys}
\text{for }x_\sigma\leq k<y_\sigma,\quad \text{ we must have }s_\sigma(k)=k+1.
\end{equation}
Indeed, to have $s_\sigma(k)>k+1$ 
would contradict the definition of $x_\sigma$, and to have $s_\sigma(k)=k$ would 
contradict the definition of $y_\sigma$. 
%
It follows that $\sigma(k)=k+1$ for 
$x_\sigma<k<y_\sigma$. By the pigeonhole principle, 
$\sigma(y_\sigma)\leq x_{\sigma}+1$.  But if $\sigma(y_\sigma)$ were $x_{\sigma}+1$ we 
could not have $s_\sigma(x_\sigma)=x_\sigma+1=s_\sigma(x_\sigma-1)$, so in fact $\sigma(y_\sigma)\leq x_\sigma$. 
For the same reason, we must have 
$\sigma(x_\sigma)\leq x_\sigma$.

Since $\sigma(x_\sigma)\leq x_\sigma$ and $\sigma(y_\sigma)\leq x_\sigma$ while $\sup\sigma([k])=s_\sigma(k)=k+1>x_\sigma$ for $x_\sigma\leq k<y_\sigma$, we see that neither $\sigma(x_\sigma)$ nor 
$\sigma(y_\sigma)$ ever realizes this supremum. Therefore exchanging $\sigma(x_\sigma)$ and $\sigma(y_\sigma)$ will not change $\sup \sigma([k])$ for any $k$. This shows that $s_\sigma=s_{\overline{\sigma}}$, proving Claim~\ref{claim:invol}.
\end{proof}

\begin{claims}
\label{claim:dealbad}
We have $\psi_{\ast}([A_{\II}]^{\text{bad}}) = 0$.
\end{claims}
\begin{proof}[Proof of Claim~\ref{claim:dealbad}]
Let $\sigma \in S_n$ be a bad permutation.  By definition, we have
$(-1)^{\overline{\sigma}} = (-1) \cdot (-1)^{\sigma}$, so it is enough to prove
that that the chambers $\psi(\VV_{\sigma})$ and $\psi(\VV_{\overline{\sigma}})$ coincide. 
In other words, we must show that
$[U_{\sigma([k])}]=[U_{\overline{\sigma}([k])}]$ for all $1\leq k\leq n-1$.
For all $k<x_\sigma$ or $k\geq y_\sigma$, we have $(x_\sigma\ \,y_\sigma)\big([k])=[k]$, so 
$\overline{\sigma}([k])=\sigma([k])$ and the claim is automatic.  For the remaining 
$x_\sigma\leq k<y_\sigma$, we showed in~\eqref{eq:xsys} that $s_\sigma(k)=s_{\overline{\sigma}}(k)=k+1$. 
By~\eqref{eq:clVsigma}, for such $k$ we have 
$[U_{\sigma([k])}]=a_k=[U_{\overline{\sigma}([k])}]$. This concludes the proof of Claim~\ref{claim:dealbad}.
\end{proof}

By definition $[A_{\II}] = [A_{\II}]^{\text{good}} + [A_{\II}]^{\text{bad}}$, and Claims~\ref{claim:dealgood} and~\ref{claim:dealbad} state that 
\[\psi_{\ast}([A_{\II}]^{\text{good}}) = [A_{\bS}] \quad \text{and} \quad \psi_{\ast}([A_{\II}]^{\text{bad}}) = 0,\]
so $\psi_{\ast}([A_{\II}]) = [A_{\bS}]$, as desired. This concludes the proof of Proposition~\ref{prop:foldedapartment}.
\end{proof}

\subsection{Proof of the Non-Vanishing and Non-Integrality Theorems}
\label{section:nonvanishingnonintegrality}

In this section we prove Theorems~\ref{maintheorem:nonvanishing} and~\ref{maintheorem:nonintegrality}. In fact, we prove the stronger Theorems~\ref{thm:nonvanishingstronger} and~\ref{thm:nonintegralitystronger}, which apply to the special automorphism groups of \emph{arbitrary} (not necessarily free) finite rank projective modules.  If $M$ is a rank $n$ projective $\O_K$-module, its special automorphism group $\SL(M)$ is also a lattice in the semisimple Lie group $\SL(M\otimes \R)$, commensurable with 
but not in general isomorphic to $\SL_n\O_K$.  Since $\SL(M)$ is commensurable with $\SL_n \O_K$, its rational
cohomological dimension equals $\SLvcd = \vcd(\SL_n\O_K)$.



\begin{theorem-prime}{maintheorem:nonvanishing}
\label{thm:nonvanishingstronger}
Let $K$ be a number field with ring of integers $\O_K$ and let $M$ be a rank $n$ projective $\O_K$-module. Then for all $n\geq 2$,
\[\dim \HH^{\SLvcd}(\SL(M);\Q)\geq (\abs{\class(\O_K)}-1)^{n-1}.\]
\end{theorem-prime}
 Theorem~\ref{maintheorem:nonvanishing}
is the special case of Theorem~\ref{thm:nonvanishingstronger} with $M \iso \O_K^n$. 
\begin{proof}
Define $\St(M)\coloneq \widetilde{\HH}_{n-2}(\Tits(M);\Z)$. 
Borel--Serre's work in \cite{BorelSerreCorners} applies to $\SL(M)$ and shows that it satisfies Bieri--Eckmann duality
with rational dualizing module $\St(M)$ (see \S \ref{section:structureTitsM} for the definition of this; 
the key point here is that $\SL(M) \subset \SL_n(K)$ and the group
$\SL_n(K)$ acts on the Borel--Serre bordification of the symmetric space; c.f. \eqref{eqn:identifysteinberg}).
Therefore just as in Equation~\eqref{eq:dualizing}, we have
\[\HH^{\SLvcd}(\SL(M);\Q)\iso(\St(M)\tensor \Q)_{\SL(M)}.\]

Consider the map $\psi\colon \Tits(M)\to X_{n-1}(\class(\O))$ from the previous section, which is invariant under the action of $\GL(M)$ on $\Tits(M)$ and hence is also invariant under $\SL(M)$. 
Recall from Example~\ref{ex:joindiscrete} that the realization of $X_{n-1}(\class(\O))$
is homotopy equivalent to a $(\abs{\class(\O)} - 1)^{n-1}$-fold wedge of $(n-2)$-spheres, so $\widetilde{\HH}_{n-2}(X_{n-1}(\class(\O));\Z)\iso \Z^N$  with $N = (\abs{\class(\O_K)} - 1)^{n-1}$.
Proposition~\ref{prop:foldedapartment} states that the  map
\[\psi_{\ast}\colon \St(M)=\widetilde{\HH}_{n-2}(\Tits(M);\Z) \rightarrow \widetilde{\HH}_{n-2}(X_{n-1}(\class(\O));\Z)\iso \Z^N\] is surjective. 
Since $\psi$ is $\SL(M)$-invariant, this map factors through the coinvariants, yielding  a surjection $\St(M)_{\SL(M)}\onto \Z^N$. Tensoring with $\Q$ yields a surjection 
\[\St(M)_{\SL(M)}\tensor \Q\iso (\St(M)\tensor \Q)_{\SL(M)}\onto \Q^N,\]
so the dimension of $(\St(M)\tensor \Q)_{\SL(M)}$ is at least $(\abs{\class(\O_K)}-1)^{n-1}$, as claimed.
\end{proof}

\begin{example}
\label{ex:hyperbolic3manifolds}
The cohomology classes we construct in Theorem \ref{maintheorem:nonvanishing} were known classically
when $n=2$.  To illustrate this, consider a quadratic imaginary field $\Q(\sqrt{-d})$ as in
Remark \ref{rem:thmCfalseimaginary}. The Bianchi group $\SL_2 \O_d$ is a lattice in $\SL_2 \C$,
and thus acts by isometries on $\HBolic^3$.  The associated
locally symmetric space $X_d = \SL_2 \O_d \backslash \HBolic^3$ is a noncompact arithmetic $3$-dimensional hyperbolic orbifold of cohomological dimension 2. The
cusps of $X_d$ are in bijection with the $\SL_2 \O_d$-conjugacy classes of parabolic subgroups in $\SL_2 \O_d$, and thus with the $\SL_2\O_d$-orbits of $\mathbb{P}^1(\Q(\sqrt{-d}))$, and thus with the ideal class group $\class(\O_d)$.
An embedded path in $X_d$ connecting one cusp to another defines an element of the locally
finite homology group $\HH_1^{\text{lf}}(X_d;\Q)$, which is dual to $\HH^2(X;\Q)\iso \HH^2(\SL_2 \O_d;\Q)$. Since there are 
$\abs{\class(\O_d)}$ cusps, intersecting with such paths gives a $(\abs{\class(\O_d)}-1)$-dimensional projection of $\HH^2(X;\Q)\iso \HH^2(\SL_2 \O_d;\Q)$.
A similar procedure works for $\SL_2 \O_K$ for any number field $K$. However the case $n \geq 3$ is more complicated:
the cusps overlap in complicated ways, so this simple argument does not~work.
\end{example}

\noindent We now give a strengthening of Theorem~\ref{maintheorem:nonintegrality}; 
Theorem~\ref{maintheorem:nonintegrality} is the special case when $M\iso \O^n$. 
\begin{theorem-prime}{maintheorem:nonintegrality}
\label{thm:nonintegralitystronger}
Let $\O$ be a Dedekind domain with $1<\abs{\class(\O)}<\infty$, and let $M$ be a rank $n$ projective $\O$-module for some $n\geq 2$. If $\abs{\class(\O)}=2$, assume that $M$ is not the unique non-free $\O$-module of rank $2$. Then $\St(M)$ is not generated by $M$-integral apartment classes.
\end{theorem-prime}
\begin{proof}
Let $\St^{\Int}(M)$ be the subspace of $\St(M)$ spanned by $M$-integral 
apartment classes. By Proposition~\ref{prop:foldedapartment}, the map 
$\psi_{\ast}\colon \St(M) \rightarrow \widetilde{\HH}_{n-2}(X_n(\class(\O));\Z)$ is surjective.  Thus to prove that $\St^{\Int}(M)\neq \St(M)$ it is enough to prove that $\psi_{\ast}(\St^{\Int}(M))$ is a proper subspace of
$\widetilde{\HH}_{n-2}(X_{n-1}(\class(\O));\Z)$.

Consider an $M$-integral frame $\II=\{I_1,\ldots,I_n\}$ with $[I_k]=c_k$ and $M=I_1\oplus\cdots\oplus I_n$; the latter property implies that 
$c_1+\cdots+c_n=[M]\in \class(\O)$.  We claim first that $\psi_{\ast}([A_{\II}])$ only depends on $\{c_1,\ldots,c_n\}$, 
and not on the particular integral frame $\II$.  To see this, recall from Example~\ref{ex:Titssummands} that the chambers of an integral apartment $A_{\II}$ are of the form
\[\VV_\sigma=I_{\sigma(1)}\ \subsetneq\  I_{\sigma(1)}\oplus I_{\sigma(2)}\ \subsetneq
\ \cdots\ \subsetneq\  I_{\sigma(1)}\oplus\cdots\oplus I_{\sigma(n-1)}\] 
for $\sigma\in S_n$.  Given $X\subset [n]$, we have $[\bigoplus_{k\in X} I_k]=\sum_{k\in X} c_k$, so the chamber 
$\VV_\sigma$ is taken by $\psi$ to the chamber
\begin{equation}
\label{eq:Csigma}
\CC_\sigma= \big(c_{\sigma(1)},c_{\sigma(1)}+c_{\sigma(2)},\ldots,c_{\sigma(1)}+\cdots+c_{\sigma(n-1)}\big).
\end{equation}
The claim follows.

Next, assume for the moment that $c_i=c_j$ for some $i \neq j$.  Define the 
involution $\sigma\mapsto \widehat{\sigma}$ on $S_n$ by 
$\widehat{\sigma}=(i\ \,j) \sigma$.   From~\eqref{eq:Csigma} we see that 
$\CC_{\widehat{\sigma}}$ will coincide with $\CC_{\sigma}$ for all 
$\sigma\in S_n$.  Since $(-1)^{\widehat{\sigma}}=(-1)\cdot (-1)^{\sigma}$, the 
chamber $\VV_{\widehat{\sigma}}$ occurs with opposite orientation from 
$\VV_\sigma$ in $[A_{\II}]$.  We conclude that $\psi_{\ast}([A_{\II}])=0$ 
unless the classes $c_1,\ldots,c_n\in \class(\O)$ are distinct, in which case
(up to sign) $\psi_{\ast}([A_{\II}])$ only depends on the unordered set
$\{c_1,\ldots,c_n\} \subset \class(\O)$.
This already shows that when $n>\abs{\class(\O)}$, the image
$\psi_{\ast}(\St^{\Int}(M))$ vanishes. We can also note at this point that for $n=2$ and $\cl(\O)=\Z/2\Z$, there are \emph{no} 2-element subsets $\{c_1,c_2\}$ with $c_1+c_2=0\in \Z/2\Z$. 
Our assumptions say that if $n=2$ and $\cl(\O) = \Z/2\Z$, then $[M]=0$; in this case
we have that $\psi_{\ast}(\St^{\Int}(M)) = 0$.  Since
$\widetilde{\HH}_{n-2}(X_{n-1}(\class(\O));\Z) = \widetilde{\HH}_{0}(\Z/2\Z;\Z)\iso\Z$ in this case, this shows that $\psi_{\ast}(\St^{\Int}(M))$ is a proper subspace.

We complete the proof in all other cases by an elementary counting argument. By Example~\ref{ex:joindiscrete}, the rank of
$\widetilde{\HH}_{n-2}(X_{n-1}(\class(\O));\Z)$ is $(\abs{\class(\O)}-1)^{n-1}$.  The precise number of 
$n$-element subsets $\{c_1,\ldots,c_n\}\subset \class(\O)$ satisfying 
$c_1+\cdots+c_n=[M]$ is hard to calculate and depends on the group structure of $\class(\O)$, but we can easily compute an upper bound that will suffice.
For any abelian group $A$ and any fixed $a\in A$, the number of \emph{ordered} tuples $(c_1,\ldots,c_n)$ with $c_1+\cdots+c_n=a$ is precisely $\abs{A}^{n-1}$, so the number of ordered tuples with \emph{distinct} entries is at most $\abs{A}^{n-1}$. The number of $n$-element subsets $\{c_1,\ldots,c_n\}$ with $c_1+\cdots+c_n=a$ is thus at most $\frac{\abs{A}^{n-1}}{n!}$.  
To prove that $\psi_{\ast}(\St^{\Int}(M))$ is a proper subspace of
$\widetilde{\HH}_{n-2}(X_{n-1}(\class(\O));\Z)$,
it is therefore enough to prove that
\[\frac{\abs{\class(\O)}^{n-1}}{n!} < (\abs{\class(\O)}-1)^{n-1}\]
for all $n \geq 2$ and $\abs{\class(\O)} \geq 2$ except for the single case
$n = 2$ and $\abs{\class(\O)} = 2$ (which we handled above).

In fact, we will prove that for $x \in \R$ with $x \geq 2$ and $n \geq 2$, we have
\begin{equation}
\label{eq:real-bound}
\frac{x^{n-1}}{n!}< (x-1)^{n-1}
\end{equation}
except when $x=2$ and $n=2$.
Indeed, \eqref{eq:real-bound} can be rearranged to $\big(\frac{x}{x-1})^{n-1}<n!$. When $x>2$ we have $\frac{x}{x-1}<2$, and $n!=2\cdot 3\cdots n$ is the product of $n-1$ terms, each of which is at least $2$. When $x=2$, the desired inequality is $2^{n-1}<n!$, which follows when $n>2$ by noting that $2^{n-1}<2\cdot 3^{n-2}\leq n!$ in this case.
\end{proof}
In the excluded case when $\cl(\O)\iso \Z/2\Z$ and $[M]\neq 0$, let $I\subset \O$ be a non-principal ideal, so that $M\iso \O\oplus I$. It can be verified as in Proposition~\ref{prop:basecasen2} that $\St(M)$ is spanned by $M$-integral apartment classes if and only if the group $\SL(M)$ is generated by diagonal matrices together with elementary matrices of the form ${\scriptstyle \begin{pmatrix}1&I\\0&1\end{pmatrix}}$ and ${\scriptstyle \begin{pmatrix}1&0\\I^{-1}&1\end{pmatrix}}$. It would be interesting to know when this  condition holds for $\SL(M)$; we are not aware of any results on this subject.

\begin{footnotesize}
\noindent
\begin{tabular*}{\linewidth}[t]{@{}p{\widthof{Department of Mathematics}+0.35in}@{}p{\widthof{Department of Mathematics}+0.35in}@{}p{\linewidth - \widthof{Department of Mathematics} - \widthof{Department of Mathematics}-0.7in}@{}}
{\raggedright
Thomas Church\\
Department of Mathematics\\
Stanford University\\
450 Serra Mall\\
Stanford, CA 94305\\
\myemail{tfchurch@stanford.edu}}
&
{\raggedright
Benson Farb\\
Department of Mathematics\\
University of Chicago\\
5734 University Ave.\\
Chicago, IL 60637\\
\myemail{farb@math.uchicago.edu}}
&
{\raggedright
Andrew Putman\\
Department of Mathematics\\
University of Notre Dame\\
279 Hurley Bldg\\
Notre Dame, IN 46556\\
\myemail{andyp@nd.edu}}
\end{tabular*}\hfill
\end{footnotesize}


\begin{thebibliography}{vdKL}
\begin{footnotesize}
\setlength{\itemsep}{1pt}


\bibitem[AR]{AshRudolph}
A. Ash and L. Rudolph, The modular symbol and continued fractions in higher dimensions, \emph{Invent. Math.} 55 (1979), no.~3, 241--250. \doi{10.1007/bf01406842}.

\bibitem[BMS]{BassMilnorSerre}
H. Bass, J. Milnor, and J.-P. Serre, Solution of the congruence subgroup problem for $\SL_n\,(n\geq 3)$ and $\Sp_{2n}\,(n\geq 2)$, \emph{Inst. Hautes \'Etudes Sci. Publ. Math.} No. 33 (1967), 59--137. \doi{10.1007/bf02684586}.

\bibitem[Be]{BestvinaPL}
M. Bestvina, PL Morse theory, \emph{Math. Commun.} \textbf{13} (2008), no.~2, 149--162.  Available at\hfill\linebreak \href{http://hrcak.srce.hr/file/48930}{\nolinkurl{hrcak.srce.hr/file/48930}}

\bibitem[BE]{BieriEckmannDuality}
R. Bieri and B. Eckmann, Groups with homological duality generalizing Poincar\'e duality, \emph{Invent. Math.} 20 (1973), 103--124. \doi{10.1007/bf01404060}.

\bibitem[Bo]{BorelStable}
A. Borel, Stable real cohomology of arithmetic groups, \emph{Ann. Sci. \'Ecole Norm. Sup.} (4) 7 (1974), 235--272. Available at \href{http://www.numdam.org/item?id=ASENS_1974_4_7_2_235_0}{\nolinkurl{www.numdam.org/item?id=ASENS_1974_4_7_2_235_0}}

\bibitem[BS]{BorelSerreCorners}
A. Borel and J.-P. Serre, Corners and arithmetic groups, \emph{Comment. Math. Helv.} 48 (1973), 436--491. \doi{10.1007/bf02566134}.

\bibitem[Br]{BrownBuildings}
K. S. Brown, \textit{Buildings}, Springer, New York, 1989. Freely available at\newline
\url{http://www.math.cornell.edu/~kbrown/buildings/}

\bibitem[C]{Cohn}
P. M. Cohn, On the structure of the ${\rm GL}\sb{2}$ of a ring, \emph{Inst. Hautes \'Etudes Sci. Publ. Math.} No. 30 (1966), 5--53. \doi{10.1007/bf02684355}.

\bibitem[DP]{DayPutmanComplex}
M. Day and A. Putman, The complex of partial bases for $F_n$ and finite generation of the Torelli subgroup of $\Aut(F_n)$, \emph{Geom. Dedicata} 164 (2013), 139--153. \doi{10.1007/s10711-012-9765-6}. \arXiv{1012.1914}.

\bibitem[F1]{FrankeMain} J. Franke, Harmonic analysis in weighted $L^2$-spaces, \emph{Ann. Sci. \'Ecole Norm. Sup.} (4)  31  (1998),  no. 2, 181--279. \doi{10.1016/s0012-9593(98)80015-3}.

\bibitem[F2]{FrankeTopological} J. Franke, A topological model for some summand of the Eisenstein cohomology of
 congruence subgroups,
in \emph{Eisenstein series and applications}, 
 27--85, \emph{Progr. Math.}, 258, Birkh\"auser Boston, Boston, MA,  2008. \doi{10.1007/978-0-8176-4639-4_2}.

\bibitem[FS]{FrankeSchwermer} J. Franke and J. Schwermer, A decomposition of spaces of automorphic forms, and the Eisenstein
 cohomology of arithmetic groups,
 \emph{Math. Ann.}  311  (1998),  no. 4, 765--790. \doi{10.1007/s002080050208}.

\bibitem[G]{Geller}
S. Geller, On the $GE_n$ of a ring,
 \emph{Illinois J. Math.}  21  (1977),  no. 1, 109--112. \newline \href{http://projecteuclid.org/euclid.ijm/1256049506}{\nolinkurl{projecteuclid.org/euclid.ijm/1256049506}}

\bibitem[LS]{LeeSzczarbaCongruence}
R. Lee and R. H. Szczarba, On the homology and cohomology of congruence subgroups, \emph{Invent. Math.} 33 (1976), no.~1, 15--53. \doi{10.1007/bf01425503}.

\bibitem[L]{Liehl}
B. Liehl, On the group $SL_2$ over orders of arithmetic type,
 \emph{J. Reine Angew. Math.}  323  (1981), 153--171. \doi{10.1515/crll.1981.323.153}.
 
\bibitem[vdK]{VanDerKallenStability}
W. van der Kallen, Homology stability for linear groups, \emph{Invent. Math.} 60 (1980), no.~3, 269--295. \doi{10.1007/bf01390018}.

\bibitem[vdKL]{LooijengaVanDerKallen}
W. van der Kallen and E. Looijenga, Spherical complexes attached to symplectic lattices, \emph{Geom. Dedicata} 152 (2011), 197--211. \doi{10.1007/s10711-010-9553-0}. \arXiv{1001.0883}

\bibitem[Ma]{MaazenThesis}
H. Maazen, Homology Stability for the General Linear Group, thesis, University of Utrecht (1979). Available at \href{http://www.persistent-identifier.nl/?identifier=URN:NBN:NL:UI:10-1874-237657}{\nolinkurl{www.persistent-identifier.nl/?identifier=URN:NBN:NL:UI:10-1874-237657}}
\change

\bibitem[Mi]{MilnorKtheory}
J. Milnor,  \emph{Introduction to algebraic $K$-theory},
Annals of Mathematics Studies  72,
Princeton University Press, Princeton, N.J.; University of Tokyo Press,
 Tokyo,  1971. \doi{10.1515/9781400881796}.

\bibitem[PS]{PutmanStudenmund}
A. Putman and D. Studenmund, The dualizing module and top-dimensional cohomology group of $\GL_n \O_K$, in preparation.

\bibitem[Q]{QuillenPoset}
D. Quillen, Homotopy properties of the poset of nontrivial $p$-subgroups of a group, \emph{Adv. in Math.} 28 (1978), no.~2, 101--128. \doi{10.1016/0001-8708(78)90058-0}.

\bibitem[Ra]{RahmTorsion}
A. D. Rahm, The homological torsion of $\rm{PSL}\sb 2$ of the imaginary quadratic integers, \emph{Trans. Amer. Math. Soc.} 365 (2013), no.~3, 1603--1635. \doi{10.1090/s0002-9947-2012-05690-x}. \arXiv{1108.4608}.

\bibitem[Re]{ReinerUnimodular}
I. Reiner, Unimodular complements, \emph{Amer. Math. Monthly} \textbf{63} (1956), 246--247. \doi{10.2307/2310351}.

\bibitem[Se]{SerreTrees}
J.-P. Serre, \textit{Trees}, translated from the French by John Stillwell, Springer, Berlin, 1980. \doi{10.1007/978-3-642-61856-7}.

\bibitem[So]{SolomonChar}
L. Solomon, The Steinberg character of a finite group with $BN$-pair, in {\it Theory of Finite Groups (Symposium, Harvard Univ., Cambridge, Mass., 1968)}, 213--221, Benjamin, New York. 

\bibitem[Va]{VasersteinGen}
L. N. Vaser\v ste\u\i n, On the group $\SL_2$ over Dedekind rings of arithmetic type, \emph{Math. USSR Sbornik} 18 (1972) 2, 321--332. \doi{10.1070/SM1972v018n02ABEH001775}.

\bibitem[Vo]{VogtmannBianchi}
K. Vogtmann, Rational homology of Bianchi groups, \emph{Math. Ann.} 272 (1985), no.~3, 399--419. \doi{10.1007/bf01455567}.

\bibitem[Wa]{Wagoner}
J. B. Wagoner, Stability for homology of the general linear group of a local ring, \emph{Topology} 15  (1976),  no.~4, 417--423. \doi{10.1016/0040-9383(76)90035-5}.

\bibitem[We]{WeinbergerEuclidean}
P. J. Weinberger, On Euclidean rings of algebraic integers, in \textit{Analytic number theory (Proc. Sympos. Pure Math., Vol. XXIV, St. Louis Univ., St. Louis, Mo., 1972)}, 321--332, Amer. Math. Soc., Providence, RI. \doi{10.1090/pspum/024/0337902}.

\end{footnotesize}
\end{thebibliography}
\end{document}